\documentclass[]{amsart}

\makeatletter
\@addtoreset{equation}{section}

\makeatother

\newtheorem{theorem}{Theorem}[section]
\newtheorem{proposition}[theorem]{Proposition}
\newtheorem{corollary}[theorem]{Corollary}
\newtheorem{lemma}[theorem]{Lemma}
\theoremstyle{definition}
\newtheorem{definition}[theorem]{Definition}

\theoremstyle{remark}
\newtheorem{remark}[theorem]{Remark}

\DeclareMathOperator{\cn}{cn}

\DeclareMathOperator{\am}{am}

\usepackage{braket} 
\usepackage{amssymb} 
\usepackage{color}
\usepackage[pdftex]{graphicx}
\usepackage{mathrsfs} 
\usepackage{url}
\usepackage{hyperref}

\newcommand{\R}{\mathbf{R}}
\newcommand{\Z}{\mathbf{Z}}

\newcommand{\N}{\mathbf{N}}

\newcommand{\vd}{\delta}

\newcommand{\B}{\mathcal{B}}
\newcommand{\K}{\mathrm{K}}
\newcommand{\Ac}{\mathcal{A}_\mathrm{clamp}} 
\newcommand{\Ap}{\mathcal{A}_\mathrm{pin}}
\newcommand{\Aclosed}{\mathcal{A}_{\rm closed}}
\newcommand{\F}{\mathcal{F}}  

\begin{document}

\title[Stable and minimal elastic curves]{General rigidity principles for stable and minimal elastic curves}
\author{Tatsuya Miura}
\address[T.~Miura]{Department of Mathematics, Tokyo Institute of Technology, 2-12-1 Ookayama, Meguro-ku, Tokyo 152-8551, Japan}
\email{miura@math.titech.ac.jp}

\author{Kensuke Yoshizawa}
\address[K.~Yoshizawa]{Institute of Mathematics for Industry, Kyushu University, 744 Motooka, Nishi-ku, Fukuoka 819-0395, Japan; 
present address: Faculty of Education, Nagasaki University, 1-14 Bunkyo-machi, Nagasaki, 852-8521, Japan}
\email{k-yoshizaw@nagasaki-u.ac.jp
}

\keywords{Stability; elastica; $p$-elastica; boundary value problem; bending energy.}
\subjclass[2020]{49Q10, 53A04.}

\date{\today}

\begin{abstract}
  For a wide class of curvature energy functionals defined for planar curves under the fixed-length constraint, we obtain optimal necessary conditions for global and local minimizers.
  Our results extend Maddocks' and Sachkov's rigidity principles for Euler's elastica by a new, unified and geometric approach.
  This in particular leads to complete classification of stable closed $p$-elasticae for all $p\in(1,\infty)$ and of stable pinned $p$-elasticae for $p\in(1,2]$.
  Our proof is based on a simple but robust `cut-and-paste' trick without computing the energy nor its second variation, which works well for planar periodic curves but also extends to some non-periodic or non-planar cases.
  An analytically remarkable point is that our method is directly valid for the highly singular regime $p\in(1,\frac{3}{2}]$ in which the second variation may not exist even for smooth variations.
\end{abstract}

\maketitle

\setcounter{tocdepth}{1}
\tableofcontents

\section{Introduction}\label{sectintroduction}\label{sectintroduction}

In variational theory it is commonly important to detect locally or globally minimal critical points in order to obtain practical solutions or effective inequalities.
Here and hereafter, as usual, an element $x\in X$ is called \emph{global minimizer} (resp.\ \emph{local minimizer}) of a functional $F:X\to[0,\infty]$ if $F(x)\leq F(x')$ holds for all $x'\in X$ (resp.\ $x'\in U$, where $U$ is some neighborhood of $x$ in $X$).
We also often refer to global and local minimality as \emph{minimality} and \emph{stability}, respectively.

In this paper we focus on variational problems involving curvature of planar curves.
More precisely we consider an energy functional of the form
\begin{equation}\notag
    \F(\gamma) := \int_0^L f\big( |k| \big) ds
\end{equation}
defined for planar curves $\gamma:[0,L]\to\R^2$, where $s$ denotes the arclength parameter, $k$ the signed curvature, and $f:[0,\infty)\to[0,\infty]$ a nonnegative Borel function with additional properties specified later.

Our class of functionals will cover at least the standard \emph{bending energy} $f(x)=x^2$ as well as the \emph{$p$-bending energy} $f(x)=x^p$ for $p\in(1,\infty)$.
The bending energy is introduced by D.\ Bernoulli and studied by L.\ Euler in the 18th century 
(see \cite{Tru83, Lev} 
for the history and 
\cite{LLbook, APbook} 
for physical backgrounds) but still studied extensively from mathematical points of view; see e.g.\ \cite{Sin, Sac08_2, Miura20, MPP21} and references therein.
The $p$-bending energy has interesting analytic and geometric features in its own right and also appears in various contexts; see e.g.\ \cite{nabe14, FKN18, LP20, NP20, OPW20, Poz20, SW20, BHV, BVH, LP22, Poz22, GPT23, OW23, DMOY, MYarXiv2203, MYarXiv2209}.
The case of polynomial $f$ was also studied \cite{AGM, Hua04} (see also \cite{Garay08}).
A relevant model is used for analyzing DNA cyclization \cite{DSSVV05}.
Throughout this paper, we will impose the fixed-length constraint on curves, which is indispensable in the standard theory of bending rods or plates.
The type of boundary condition is also important in the definition of stability and minimality.
We will address not only the typical \emph{clamped boundary condition}, which fixes the endpoints up to first order, but also the zeroth-order counterpart called \emph{pinned boundary condition}.

The typical cases to which our theory applies are thus the classical problem of Euler's \emph{elastica} 
\cite{Euler} ($p=2$) as well as its $L^p$-counterpart called \emph{$p$-elastica} \cite{nabe14, LP22, MYarXiv2203, MYarXiv2209}, which is defined as a fixed-length critical point of the $p$-bending energy
$$\B_p[\gamma]:=\int_0^L|k|^pds.$$
On the level of critical points, if $p=2$, all solutions are classified by Euler and later parametrized in terms of Jacobian elliptic integrals and functions by Saalsch\"utz (cf.\ \cite{Lev,Love}).
The case of $p\neq2$ is known be much more delicate mainly because the Euler--Lagrange equation involves a nonlinear diffusion term which yields various new phenomena such as the loss of regularity.
However, the authors recently succeeded in extending the above classification to planar $p$-elasticae for a general power $p\in(1,\infty)$ by using known $p$-elliptic integrals and introducing new $p$-elliptic functions \cite{MYarXiv2203}.
Going into boundary value problems, one usually aims at (i) finding all critical points, (ii) finding all global minimizers, and (iii) finding all local minimizers for given boundary data.
These three problems involve additional difficulties, which are generally independent.
For some well-prepared boundary data, the problems (i) and (ii) are completely solved even for all $p\in(1,\infty)$; see \cite{MYarXiv2203} for closed $p$-elasticae and \cite{MYarXiv2209} for pinned $p$-elasticae.
Concerning (ii) when $p=2$, see also \cite{Miura20} for straightened boundary data and \cite{MMR23} for the cuspidal case.
On the other hand, to the authors' knowledge, the problem (iii) is solved only for closed elasticae ($p=2$) \cite{LS_85, Sac12, AKS}.

In this paper we first reveal general rigidity principles (necessary conditions) induced by stability and minimality for a wide class of functionals, including the $p$-bending energy as a special example.
In particular, those results together with our previous work \cite{MYarXiv2203,MYarXiv2209} lead to complete classification of the stability of closed $p$-elasticae for $p\in(1,\infty)$ and pinned $p$-elasticae for $p\in(1,2]$, thus solving the only remaining problem (iii) above.
We also give an effective rigidity result for pinned $p$-elasticae with $p\in(2,\infty)$.
The contents of this paper were briefly announced in \cite{Miura_RIMS}.

To the authors' knowledge, our study provides the first stability analysis for planar $p$-elasticae.
Recently, Gruber--P\'ampano--Toda \cite{GPT23} proved instability of closed free $p$-elasticae in the sphere by a second variation approach.
Their study is in itself very interesting but substantially different from our focus; besides the difference of the ambient spaces, our study deals with the length constraint and does not assume any restriction on the regularity nor on the inflection points of critical points, which makes the stability analysis much more delicate.
In particular, we point out (in Appendix \ref{app:second_variation}) that in the highly singular regime $p\in(1,\frac{3}{2}]$, if a $p$-elastica has an inflection point, then \emph{the second variation may not exist even for smooth variations}.
Here, instead of standard second variation arguments, we propose a new but very simple `cut-and-paste' trick, which turns out to be surprisingly robust and indeed directly works even for inflectional $p$-elasticae with $p\in(1,\frac{3}{2}]$.

Last but not least, even in the classical case $p=2$ for Euler's elastica, our approach seems to be the first attempt to explain the rigidity by global minimality and by local minimality in a unified manner.

This paper is organized as follows:
In Section~\ref{sect:main_results} we state our main results, both in terms of general principles and applications to $p$-elasticae, and explain the key idea of our proof, while mentioning some previous studies more precisely.
After short preliminaries in Section~\ref{sect:pre}, we prove our general principles in the clamped case (Theorems~\ref{thm:planar_global} and \ref{thm:planar_local}) in Section~\ref{sect:clamped}, and in the pinned case (Theorems \ref{thm:pinned2} and \ref{thm:pinned1*}) in Section~\ref{sect:pinned}.
In Section~\ref{sect:p-elastica} we discuss applications to planar $p$-elasticae.
Finally we apply our method to some spatial elasticae in Section~\ref{sect:spatial}.

\section{Main results}\label{sect:main_results}

We first very briefly review some of relevant previous studies in order to motivate our main results.

In 1906, Born developed stability theory for Euler's elastica \cite{Born}.
Among other results, he found the general principle that if a clamped planar elastica has no zero of the curvature (i.e., locally convex), then it is stable with respect to smooth perturbations.
Born also addresses more general elasticae, albeit with the help of numerical computations, and gave the first detailed comparisons with experiments.

Maddocks \cite{Maddocks1981,Maddocks1984} developed linear stability analysis for Euler's elastica (with possibly nonuniform bending rigidity), which explains physically natural (in)stability for various boundary conditions, extending some previous results cited therein.
For Euler's elastica, it is shown that `higher modes' (with many inflection points) are basically unstable for both clamped and pinned boundary conditions.
In particular, the results for the pinned boundary condition are summarized as follows:
\begin{itemize}
    \item[(M1)] The pinned first mode loses its stability when the endpoints meet.
    \item[(M2)] The pinned higher modes are always unstable.
\end{itemize}
A schematic diagram for (M1) is given in Figure \ref{fig:Maple1}.
The higher modes in (M2) correspond to the curves $\gamma_\mathrm{arc}^N,\gamma_\mathrm{loop}^N$ with $N\geq2$ in Figure \ref{fig:Maple2}.
In fact, it is rigorously known (cf.\ \cite{Love,Ydcds}) that all the pinned elasticae are classified as in Figure \ref{fig:Maple2}.
Combined with this classification, Maddocks' results imply that the embedded convex arc $\gamma^1_\mathrm{arc}$ is the only stable one (in the sense of linear stability).
Langer--Singer \cite{LS_85} developed a different approach, which not only covers spatial elasticae with special symmetry but also recovers some of Maddocks' results.

\begin{figure}[htbp]
\centering
\includegraphics[width=11cm]{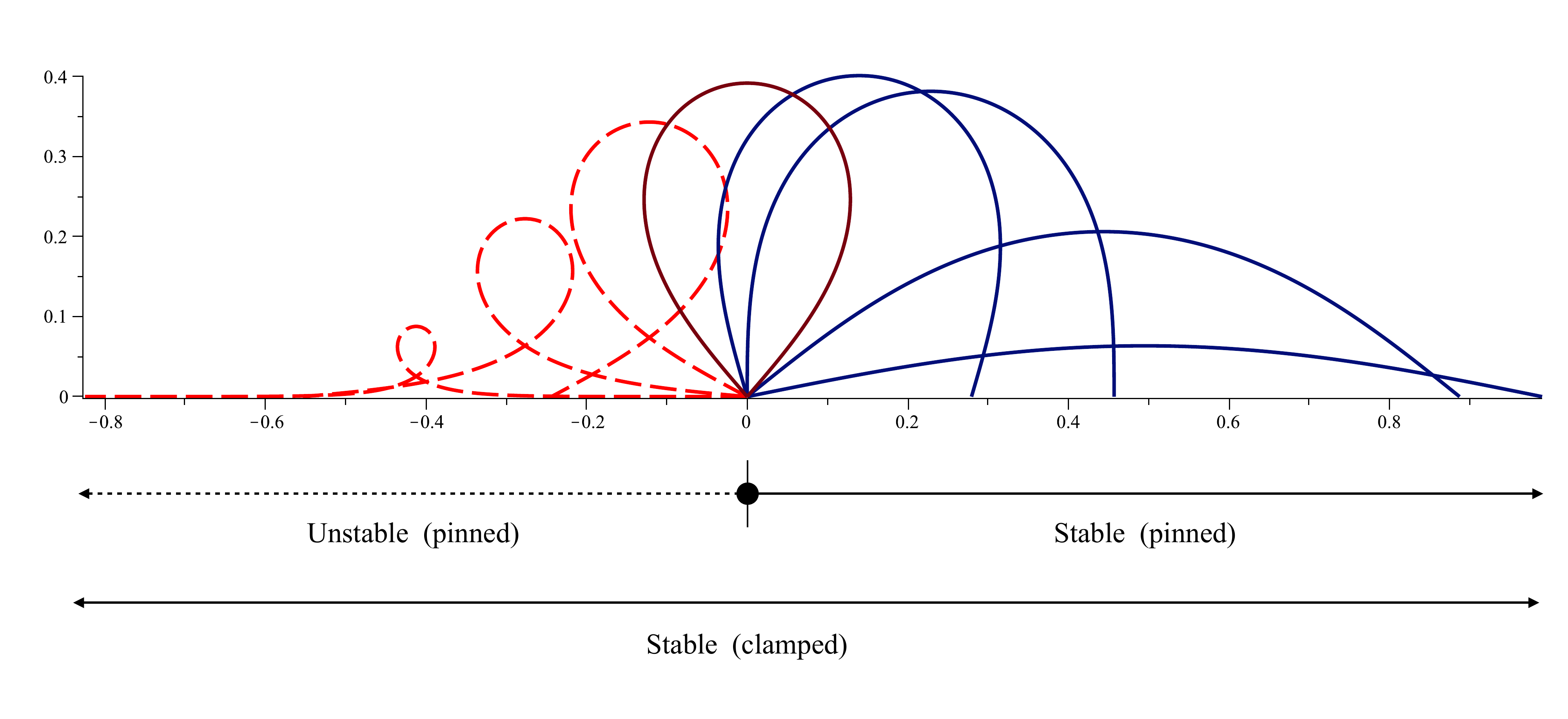} 
\caption{Stability of convex arcs and loops.}
\label{fig:Maple1}
\end{figure}
\begin{figure}[htbp]
\centering
\includegraphics[width=70mm]{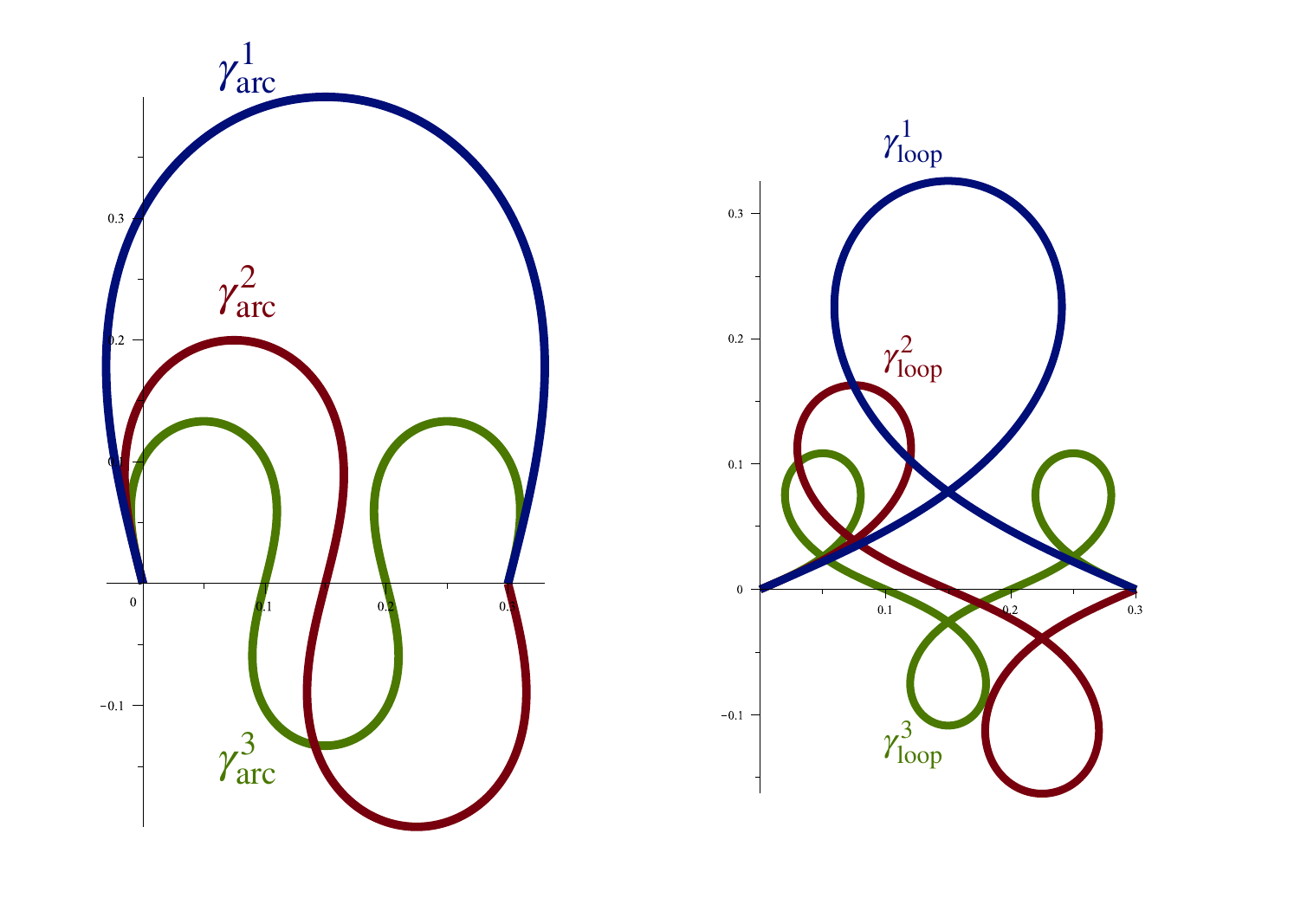} 
\caption{Pinned elasticae of arc type (left) and of loop type (right).}
\label{fig:Maple2}
\end{figure}

In a recent series of papers \cite{Sac08_2,Sac08_3, SL10, Sac12, SS14}, focusing on planar elasticae subject to the clamped boundary condition, Sachkov and his coauthors developed a more detailed theory of not only stability but also minimality by an optimal control approach.
In particular, Sachkov found optimal rigidity principles in terms of the natural periodic structure of elasticae, cf.\ statements (1) and (3.3) (or (3.5)) in \cite[Theorem 5.1]{Sac08_3}.
Here the periodicity means that the tangent direction is periodic; the curve itself is generically quasi-periodic, cf.\ Figure~\ref{fig:2-Elastiace}.
Sachkov's principles are roughly summarized as follows:
\begin{itemize}
    \item[(S1)] If a clamped elastica is minimal, then it does not exceed one period.
    \item[(S2)] If a clamped elastica is stable, then it does not contain three inflection points in its interior.
\end{itemize}
Principle (S1) is trivial for closed elasticae in view of scaling --- any closed curve has less bending energy than its multiple covering of same length --- but is not trivial for non-closed elasticae 
since simple scaling arguments do not work.
Principle (S2) is in line with Maddocks' instability result for higher modes.
Both (S1) and (S2) are optimal (cf.\ Remarks \ref{rem:optimal-global} and \ref{rem:optimal-local}) and very useful for detecting (local) minimizers.

\begin{figure}[htbp]
\centering
\includegraphics[width=125mm]{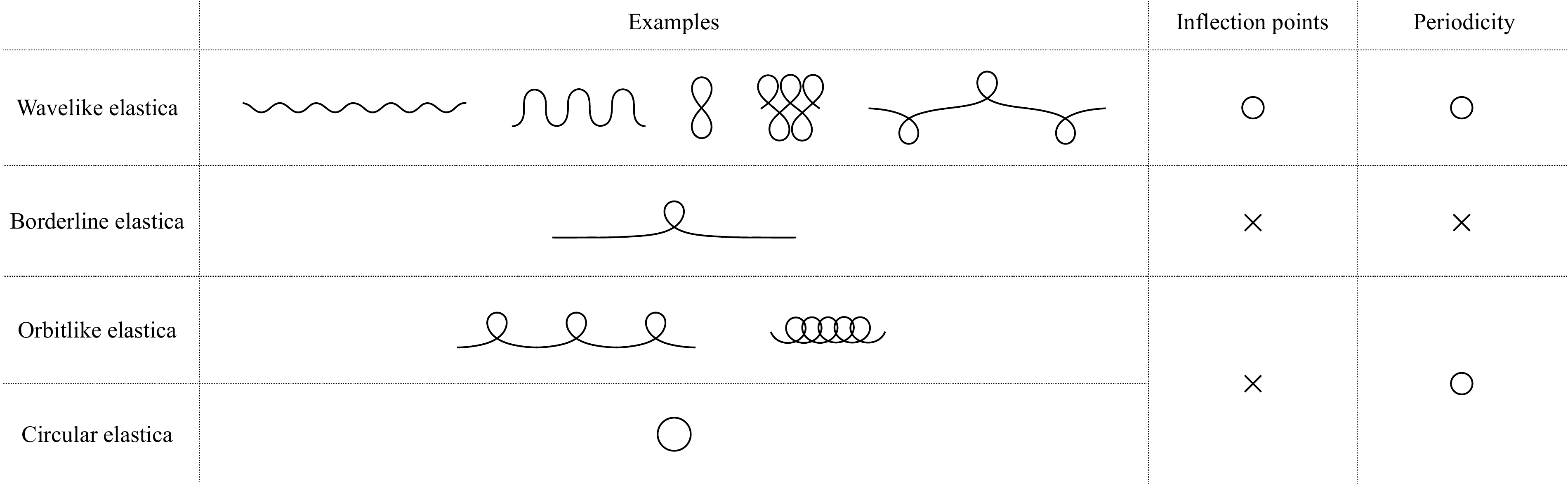}
\caption{Basic patterns of elasticae.}
\label{fig:2-Elastiace}
\end{figure}

Our main results extend all (M1), (M2), (S1), (S2) to a wide class of energy functionals (in the natural Sobolev framework).
Such extensions are highly nontrivial even for $p$-elasticae, since both Maddocks' and Sachkov's methods importantly use the fact that the curvature term is quadratic.
Indeed, Sachkov's argument depends on explicit computations only valid for the standard bending energy; Maddocks' argument is based on a rather `representation-free' linear stability analysis (which makes it possible to address non-uniform cases) but the rigorous application of linear stability analysis on the nonlinear level would be significantly delicate even for $p$-elasticae with $p\neq2$ due to the lack of Hilbert structures, which in particular leads to a generic loss of regularity.

The rest of this section proceeds as follows:
We first present general principles in the clamped case corresponding to (S1) and (S2) in Section \ref{subsect:general_clamped}, because of the simplicity of the statements and also their applicability to \emph{any} boundary condition, cf.\ Remark \ref{rem:BCfree}.
Then we turn to the pinned case corresponding to (M1) and (M2) in Section \ref{subsect:general_pinned}.
In Section \ref{subsect:p-elastica}, we collect some typical consequences in $p$-elastica theory.
Finally we explain the key idea of our proof in Section \ref{sect:trick}.

\subsection{General rigidity principles: Clamped case}\label{subsect:general_clamped}

Let $p\in(1,\infty)$, $L>0$ and $P_0,P_1,V_0,V_1\in\R^2$ with $|P_0-P_1|<L$ and $|V_0|=|V_1|=1$.
Define
$$W_\mathrm{arc}^{2,p}(0,L;\R^2):=\Set{\gamma \in W^{2,p}(0,L;\R^2) |\, |\gamma'|\equiv1 },$$
the set of planar arclength parametrized Sobolev curves of length $L$, and then
\begin{equation*}
    \Ac:=\{ \gamma\in W_\mathrm{arc}^{2,p}(0,L;\R^2) \mid \gamma(0)=P_0,\ \gamma(L)=P_1,\ \gamma'(0)=V_0,\ \gamma'(L)=V_1 \},
\end{equation*}
the admissible set subject to the clamped boundary condition, equipped with the relative topology induced by the $W^{2,p}$-norm.

The first result extends (S1) in a contrapositive form.
Our key hypothesis is the following property of regularity-improvement:
\begin{equation}\label{H1*}\tag{H1}
    \text{If $\gamma$ is a global minimizer of $\F$ in $\Ac$, then $\gamma\in C^2([0,L];\R^2)$.}
\end{equation}
Hereafter we often abbreviate $\gamma\in C^2([0,L];\R^2)$ as $\gamma\in C^2$.
This property is naturally expected whenever $f$ is not too wild, since any minimizer weakly solves the Euler--Lagrange equation.
Under this hypothesis we obtain

\begin{theorem}[Rigidity of clamped global minimizers]\label{thm:planar_global}
    Suppose \eqref{H1*} holds.
    If $\gamma\in\Ac$ has two points
    $0\leq  s_1 < s_2 \leq L$ with $ s_2-s_1<L$ such that
      \begin{align} \label{eq:0831-2}
      \gamma'(s_1)=\gamma'(s_2) \ \  \text{and} \ \  -k(s_1) \ne k(s_2),
    \end{align}
    then $\gamma$ is not a global minimizer of $\F$ in $\Ac$.
\end{theorem}

This result indeed extends (S1) because after one period ($s_2=s_1+\text{period}$) one has $\gamma'(s_1)=\gamma'(s_2)$ and $k(s_1)=k(s_2)$ so that \eqref{eq:0831-2} holds unless the curvature vanishes there.

The next result extends (S2).
To this end we will similarly suppose:
\begin{equation}\label{H1**}\tag{H1'}
    \text{If $\gamma$ is a local minimizer of $\F$ in $\Ac$, then $\gamma\in C^2$.}
\end{equation}
In addition, we need to use a certain well-periodic structure.

\begin{definition}[Well-periodic curve]\label{def:well-periodic}
We call an arclength parametrized planar curve $\gamma:[0,L]\to \R^2$ \textit{well-periodic} if $\gamma$ has continuous signed curvature $k\in C([0,L])$ of the form $k(s)=\Phi(s-s_0)$, where $s_0 \in \R$, and $\Phi\in C(\R)$ is an antiperiodic odd function with antiperiod $T>0$ such that $\Set{ s\in \R | \Phi(s)=0 } = T\Z$.
We also call $T$ the antiperiod of the well-periodic curve $\gamma$.
\end{definition}

This kind of periodicity naturally appears in the objective critical points thanks to the invariance of $\F$ with respect to the change of the sign of curvature as well as the orientation of parameter.
For this well-periodic class we obtain

\begin{theorem}[Rigidity of clamped local minimizers]\label{thm:planar_local}
    Suppose \eqref{H1**} holds.
    If $\gamma\in\Ac$ is well-periodic, and has three points $0\leq s_1<s_2<s_3\leq L$ with $s_3-s_1<L$ such that
    \begin{align} \label{Yeq:0417.2}
     k(s_1)=k(s_2)=k(s_3)=0,
    \end{align}
    then $\gamma$ is not a local minimizer of $\F$ in $\Ac$.
\end{theorem}

This result directly extends (S2), and gives an effective rigidity for `inflectional' critical points.

\begin{remark}
    Theorem \ref{thm:planar_global} (resp.\ Theorem \ref{thm:planar_local}) is optimal in the sense that if $s_2-s_1=L$ (resp.\ $s_3-s_1=L$) then the assertion may fail even for classical elasticae, cf.\ Remark~\ref{rem:optimal-global} (resp.\ Remark~\ref{rem:optimal-local}).
\end{remark}

\begin{remark}\label{rem:BCfree}
    Our principles immediately extend to more general cases.
    For example, \emph{regardless} the choice of boundary conditions, the same results hold true away from the endpoints; more precisely, for a given minimal (resp.\ stable) curve $\gamma:[0,L]\to\R^2$ subject to any boundary condition, the restriction $\gamma|_{[\delta_1,L-\delta_2]}$ with $\delta_1,\delta_2>0$ must be minimal (resp.\ stable) under the clamped boundary condition and hence satisfy the above assertion.
    In addition, our principles also directly work in the case of \emph{length-penalized} problems (i.e., critical points of the functional $\F+\lambda\mathcal{L}$ among length-unconstrained admissible curves, where $\lambda\in\R$ and $\mathcal{L}$ denotes the length) as well as the case of higher codimensions (i.e., planar curves in $\R^n$ with any $n\geq2$). 
\end{remark}

\subsection{General rigidity principles: Pinned case}\label{subsect:general_pinned}

Now we turn to the pinned boundary condition.
More precisely, for $L>0$ and $P_0,P_1\in\R^2$ such that $|P_0-P_1|<L$, we consider the admissible class \begin{equation*}
    \Ap:=\{ \gamma\in W_\mathrm{arc}^{2,p}(0,L;\R^2) \mid \gamma(0)=P_0,\ \gamma(L)=P_1 \}.
\end{equation*}
In this class, a critical point often satisfies the so-called natural boundary condition that the curvature vanishes at the endpoints, thanks to the arbitrariness of first-order variations at the endpoints.
In view of this condition, it is natural to consider the following subclass of well-periodic curves, which we call \emph{$\frac{m}{2}$-fold} curves; the number $m$ corresponds to the superscripts in Figure \ref{fig:Maple2}, and roughly speaking counts the modes with the convention that `$1$-fold' means the one-period of the curvature (i.e., twice the antiperiod).

\begin{definition}[$\frac{m}{2}$-Fold well-periodic curve]\label{def:m-fold}
For $m\in \N$, we say that a well-periodic curve $\gamma:[0,L]\to \R^2$ is \emph{$\frac{m}{2}$-fold} if $k(0)=0$ and $T=L/m$ hold, where $T$ is the antiperiod of $\gamma$ defined in Definition \ref{def:well-periodic}.
\end{definition}

We first state our instability result for higher modes, extending (M2).
Here we will suppose both the regularity-improvement and the natural boundary condition:
\begin{align} \tag{H2} \label{H2}
    \begin{split}
        \text{If $\gamma$ is a local minimizer of $\mathcal{F}$ in $\Ap$, then $\gamma\in C^2$ and $k(0)=k(L)=0$.}
    \end{split}
\end{align}

\begin{theorem} \label{thm:pinned2}
Suppose \eqref{H2} holds.
Let $\gamma \in \Ap$ be a $1$-fold well-periodic curve. 
Then $\gamma$ is not a local minimizer of $\F$ in $\Ap$.
\end{theorem}

We then address the first mode, extending (M1).
In fact, our proof of this part will be based on a different mechanism from all the previous cases, and thus we will suppose a rather different type of hypothesis:
\begin{align} \tag{H3} \label{H3}
\text{
$f$ is strictly convex on $[0,\infty)$. 
}
\end{align}
Under this hypothesis we show that if a curve has a specific loop structure as in the left part of Figure \ref{fig:Maple1} or such as $\gamma^1_\mathrm{loop}$ in Figure \ref{fig:Maple2}, then it is unstable.

\begin{theorem} \label{thm:pinned1*}
Suppose \eqref{H3} holds.
Let $\gamma\in \Ap$ be a $\frac{1}{2}$-fold well-periodic curve. 
If $\gamma$ is not injective on $(0,L)$ and if $\gamma'(0) \cdot( P_1-P_0 ) >0$,
then $\gamma$ is not a local minimizer of $\F$ in $\Ap$.
\end{theorem}

Notice that the assumption automatically rules out the case that $P_0=P_1$.

\subsection{Applications to $p$-elastica}\label{subsect:p-elastica}

Now we discuss some concrete applications to $p$-elasticae under some well-chosen boundary conditions.

Our previous work ensures that for all $p\in(1,\infty)$ the $p$-bending energy $\B_p$ satisfies hypotheses \eqref{H1*}, \eqref{H1**}, \eqref{H2}, \eqref{H3}; see \cite{MYarXiv2203} for $C^2$-regularity in \eqref{H1*}, \eqref{H1**}, \eqref{H2}, and see \cite{MYarXiv2209} for the natural boundary condition in \eqref{H2}.
In addition, all wavelike $p$-elasticae obtained in \cite{MYarXiv2203} belong to the class of well-periodic curves.

We first address the case of closed curves, which can be regarded as a special case of the clamped boundary condition by taking $P_0=P_1$ and $V_0=V_1$.
It is shown in \cite{MYarXiv2203} that any closed planar $p$-elastica is either a circle or a figure-eight $p$-elastica, possibly multiply covered.
It is easy to deduce from H\"{o}lder's inequality and Fenchel's theorem that a global minimizer is a once covered circle.
Here we apply Theorem \ref{thm:planar_local} to extend the known classification of stability for $p=2$ \cite{LS_85, Sac12, AKS} to a general power $p\in(1,\infty)$.

\begin{theorem}[Stability of closed $p$-elasticae]\label{thm:classify-closed-pela}
    Let $p\in(1,\infty)$ and $\gamma\in\Ac$ with $P_0=P_1$ and $V_0=V_1$.
    Then $\gamma$ is a local minimizer of $\B_p$ in $\Ac$ if and only if $\gamma$ is either a possibly multiply covered circle or a once covered figure-eight $p$-elastica.
\end{theorem}

Since the $C^1$-regular homotopy classes of closed planar curves are exactly classified by the rotation number, and since the above local minimizers have different rotation numbers, we reach a uniqueness theorem \`{a} la Langer--Singer in $\R^2$ \cite{LS_85}.

\begin{corollary}\label{cor:LangerSinger-p}
    For each $C^1$-regular homotopy class of immersed circles in $\R^2$, stable closed planar $p$-elasticae are unique up to similarity and reparametrization.
\end{corollary}

We then turn to the pinned boundary condition.
We call a critical point of $\B_p$ in $\Ap$ \emph{pinned $p$-elastica}.
In \cite{MYarXiv2209} the authors have classified all pinned $p$-elasticae and proved uniqueness of global minimizers.
In general, there are infinitely many pinned $p$-elasticae, which are either of wavelike or flat-core type.

We first completely classify the stability of wavelike pinned $p$-elasticae.
It is shown in \cite{MYarXiv2209} that any wavelike pinned $p$-elasticae is classified as an $\frac{m}{2}$-fold well-periodic curve as in Figure \ref{fig:Maple2}, to which Theorems \ref{thm:pinned2} and \ref{thm:pinned1*} are applicable.
It turns out that all but the global minimizer is unstable.

\begin{theorem}[Stability of wavelike pinned $p$-elasticae]
\label{thm:wavelike_unique}
    Let $p\in(1,\infty)$.
    Suppose that $\gamma\in\Ap$ is a wavelike $p$-elastica.
    Then $\gamma$ is a local minimizer of $\B_p$ in $\Ap$ if and only if $\gamma$ is a global minimizer (convex arc).
\end{theorem}

The above wavelike result already yields strong rigidity.
Indeed, it is also shown in \cite{MYarXiv2209} that if $\frac{|P_0-P_1|}{L}< \frac{1}{p-1}$, then any pinned $p$-elastica is wavelike.
Hence we reach the following uniqueness of stable pinned $p$-elasticae.

\begin{corollary}\label{cor:pinned_unique}
    Suppose either that $p\in(1,2]$, or that $p\in(2,\infty)$ and $\frac{|P_0-P_1|}{L}< \frac{1}{p-1}$.
    Then stable pinned planar $p$-elasticae in $\Ap$ are unique up to isometries.
\end{corollary}

For $p\in(1,2]$, our result covers the full range of pinned boundary data.
This not only extends Maddocks' linear stability analysis as in Figure \ref{fig:Maple1} from $p=2$ to $p\in(1,2]$, but also directly justifies it on the nonlinear level (even for $p=2$).
In particular, our result could possibly be the first explicit statement on the uniqueness of stable pinned elasticae.
At any rate, the result is completely new for $p\neq2$.

We further make progress by addressing flat-core pinned $p$-elasticae, the only remaining case.
Flat-core $p$-elasticae appear only in the degenerate case $p>2$ and have peculiar non-periodicity, cf.\ Figure \ref{fig:flat-cores}, thus representing the substantial difference between the structures of $p$-elasticae and of the classical elasticae.
In particular, our previous rigidity results are based on periodicity and not directly applicable to flat-core $p$-elasticae in general.
However, the same machinery enables us to develop rather ad-hoc arguments to prove the following new instability:

\begin{theorem}[Unstable flat-core pinned $p$-elasticae]\label{thm:unstable-flatcore}
    A flat-core pinned planar $p$-elastica is unstable either if it contains two adjacent loops in opposite directions, or if a loop touches an endpoint.
\end{theorem}

\begin{center}
    \begin{figure}[htbp]
      \includegraphics[scale=0.2]{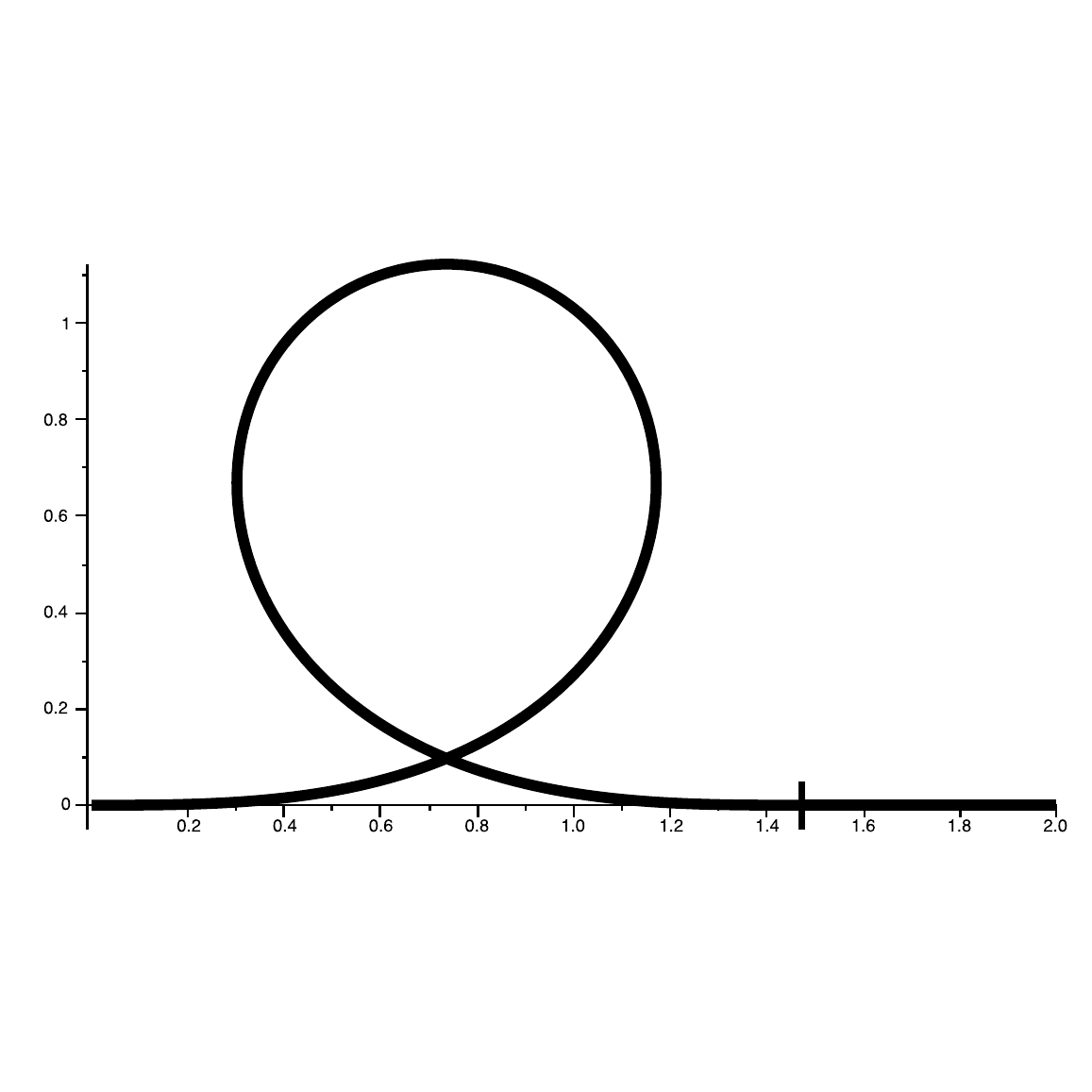}
  \hspace{5pt}
      \includegraphics[scale=0.2]{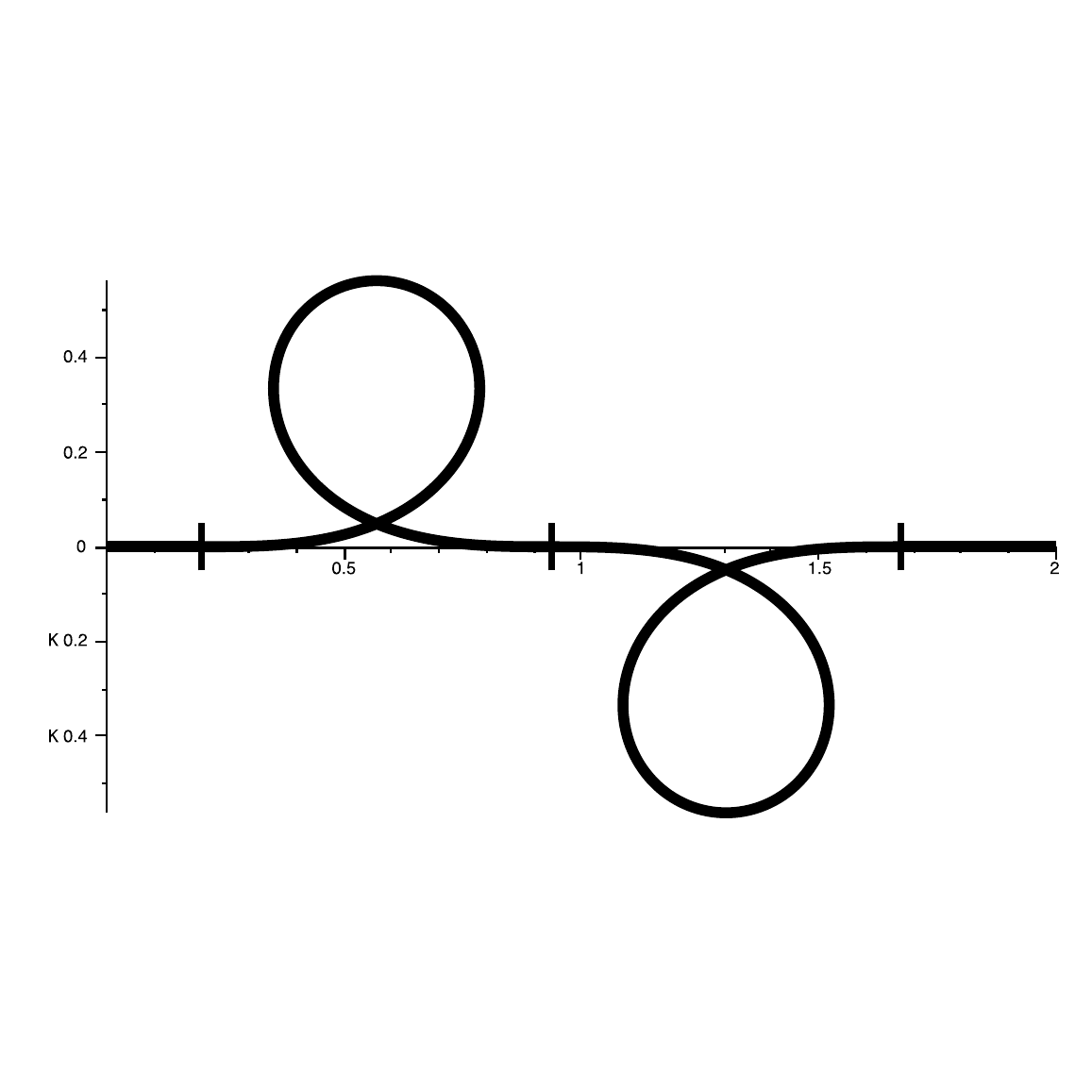}
  \hspace{5pt}
      \includegraphics[scale=0.2]{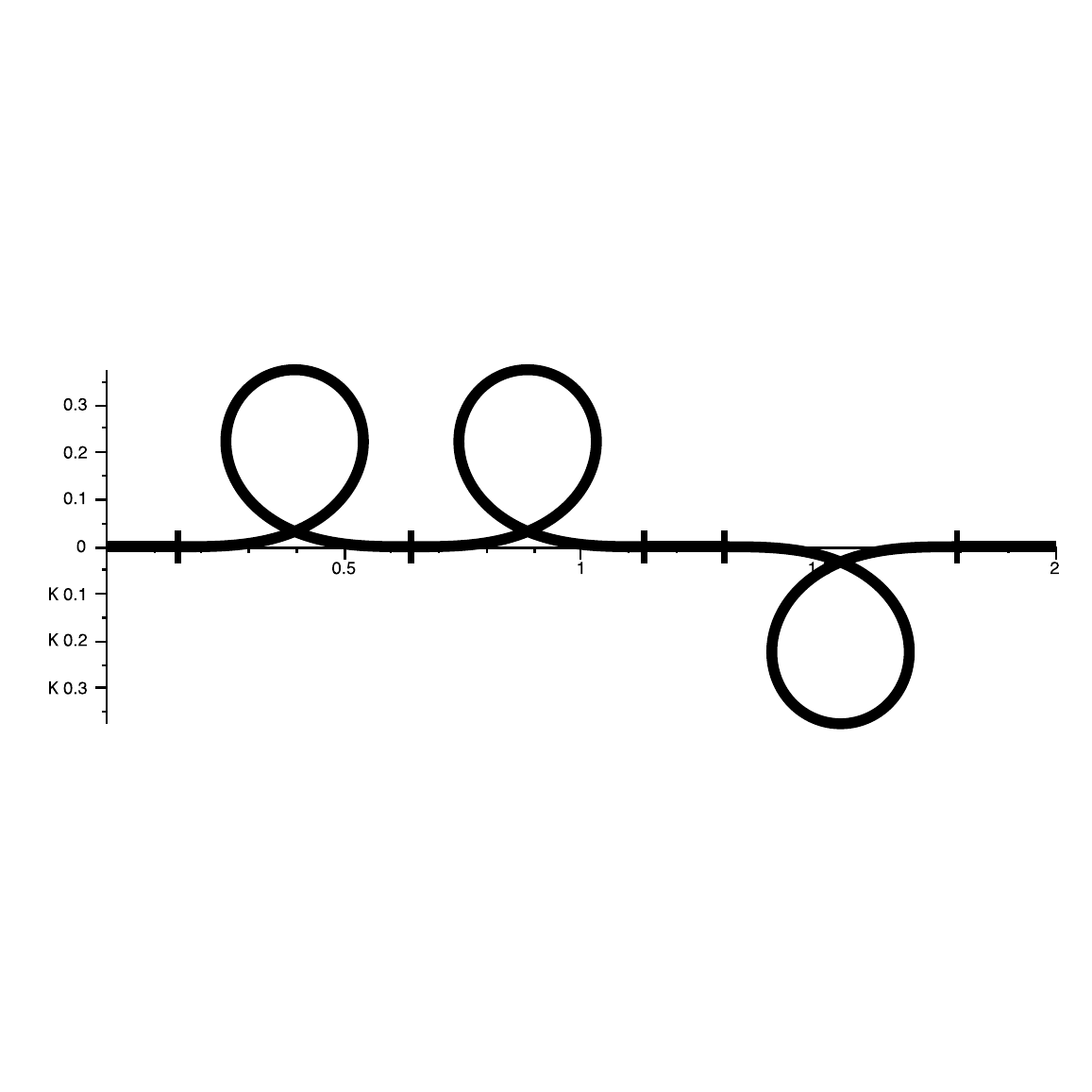}
   \vspace{-20pt}
  \caption{Flat-core $p$-elasticae.
  The left two curves are unstable (Theorem~\ref{thm:unstable-flatcore}).
  The right curve is classified as a quasi-alternating flat-core $p$-elastica (Definition~\ref{def:alternating}).}
  \label{fig:flat-cores}
  \end{figure}
\end{center}

In particular, invoking the classification in \cite{MYarXiv2209}, we obtain the following dichotomy theorem for stable pinned planar $p$-elasticae without any assumption on the exponent $p$ nor the boundary condition.
Here we call that a flat-core pinned planar $p$-elastica is \emph{quasi-alternating} if the curve is not ruled out by the above theorem, or in other words given by a certain alternating concatenation of segments and multi-loops as in Figure~\ref{fig:flat-cores} (see Definition~\ref{def:alternating} for details).

\begin{corollary}\label{cor:flatcore-dichotomy}
    Any stable pinned planar $p$-elasticae is either a global minimizer (convex arc) or a quasi-alternating flat-core $p$-elastica.
\end{corollary}

It is an interesting open problem whether there indeed exists a stable quasi-alternating flat-core $p$-elastica.
If it were true, then it would give an interesting example of a stabilization phenomenon due to degeneracy.

\subsection{A trick}\label{sect:trick}

Finally we explain the main idea that underlies most of our arguments.
All the aforementioned previous studies \cite{Born,Maddocks1981,Maddocks1984,LS_85,Sac08_2,Sac08_3, SL10, Sac12, SS14} involve explicit computations on Jacobian elliptic functions or some functional analytic structures, which are based on the quadraticity to some degree.
On the other hand, Avvakmov--Karpenkov--Sossinsky \cite{AKS} gave a new proof that multiply-covered figure-eight elasticae are unstable in the plane, by directly constructing an energy-decreasing perturbation.
This method has the strong advantage that it is almost purely based on the geometric structure of the curve.
However, their construction relies on the very special geometry of the figure-eight elastica.
A particularly important point is that their construction involves rescaling and hence does not directly extend to some non-closed curves.
In fact, such a direct construction of an energy-decreasing perturbation is often obstructed by the fixed-length constraint.

A trick we propose here to circumvent the above obstruction is a very simple indirect argument.
The key idea is to construct an \emph{equal-energy} competitor that lies in the energy space but is \emph{not regular} enough to be minimal or stable, implying a contradiction to our natural regularity hypothesis.
It turns out that the constructions of equal-energy competitors are reduced to surprisingly simple `cut-and-paste' arguments, with the help of intrinsic periodicity and symmetry of objective critical points.
We demonstrate the proof of Theorem \ref{thm:planar_global} for a special example of an elastica (but in fact the idea is completely same in the general case).
As in Figure \ref{fig_3}, for a given `more than one period' orbitlike elastica we can rotate a `one-period' part to construct a competitor which is of class $W^{2,2}$ and has the unchanged length and bending energy but cannot be a global minimizer since the discontinuity of the curvature at the cut points contradicts the $C^2$-regularity, completing the proof.
To construct local perturbations we need more structures as assumed in Theorem \ref{thm:planar_local} since the perturbations need to be close to the original curve.

\begin{figure}[htbp]
\centering
\includegraphics[width=12cm]{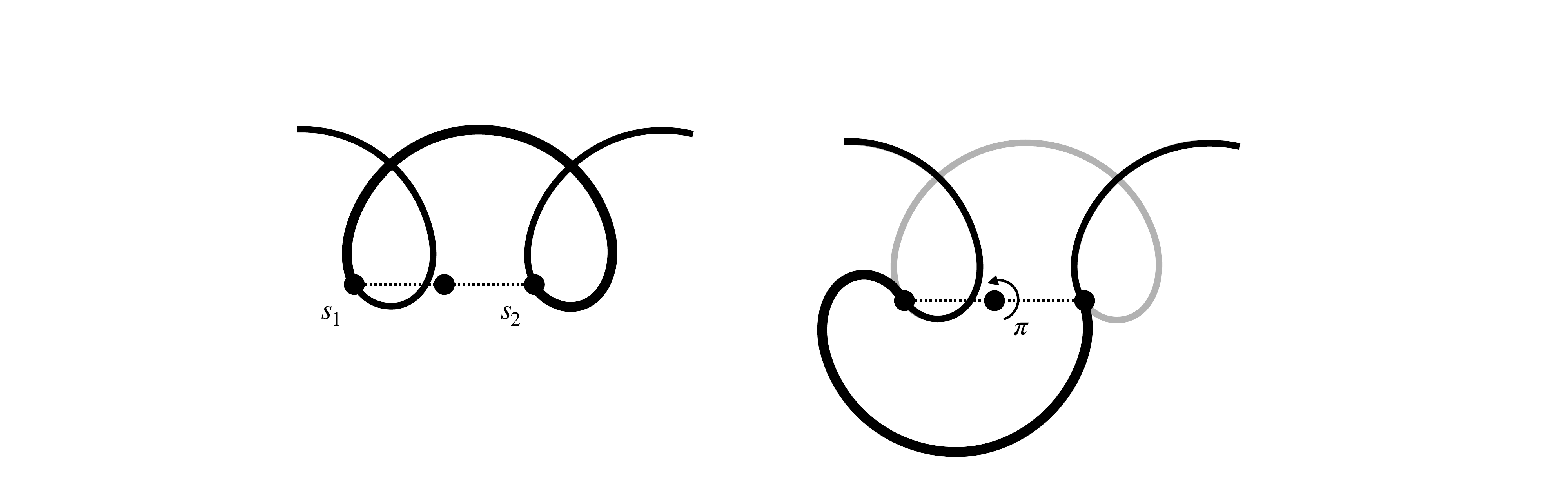} 
\caption{An original curve $\gamma$ with two points $s_1, s_2$ satisfying $\gamma'(s_1)=\gamma'(s_2)$ and $k(s_1)=k(s_2)\neq0$ (left) and the competitor $\bar{\gamma}$ with equal energy to $\gamma$ (right).}
\label{fig_3}
\end{figure}

One of the main discoveries in our study would be the fact that this simple trick is quite robust.
Although the new results presented above already provide general and optimal rigidity results for planar critical points, we also expect more.
As an example, in Section \ref{sect:spatial}, we apply our trick to elasticae in $\R^3$ to obtain old and new spatial rigidity results in a unified manner.

\section{Preliminary}\label{sect:pre}

\subsection{Concatenation}

We prepare a notation on a concatenation of curves: 
For $\gamma_j:[a_j,b_j]\to\R^2$ with $L_j:=b_j-a_j\geq0$, we define $\gamma_1\oplus\gamma_2:[0,L_1+L_2]\to\R^2$ by
\begin{align*}
  (\gamma_1\oplus\gamma_2)(s) :=
  \begin{cases}
    \gamma_1(s+a_1), \quad & s \in[0, L_1], \\
    \gamma_2(s+a_2-L_1) +\gamma_1(b_1)-\gamma_2(a_2),  & s \in [L_1,L_1+L_2], 
  \end{cases}
\end{align*}
and inductively define $\gamma_1\oplus\dots\oplus\gamma_N := (\gamma_1\oplus\dots\oplus\gamma_{N-1})\oplus\gamma_N$.
We also define 
\[
\bigoplus_{j=1}^N \gamma_j :=\gamma_1\oplus\dots\oplus\gamma_N.
\]
This notation will be useful for our cut-and-paste procedures.
Similarly, for a point $P\in \R^2$ and a curve $\gamma:[0,L]\to\R^2$, we define 
\begin{align*}
  (P\oplus\gamma)(s) := \gamma(s)-\gamma(0)+P, \quad  s \in[0, L]. 
\end{align*}

\subsection{Well-periodic curves}

Here we collect some fundamental properties of well-periodic curves in Definition \ref{def:well-periodic}.

Recall that a function $\Phi$ is said to be antiperiodic if there is $T>0$ such that $-\Phi(x) = \Phi(x+T)$ holds for all $x\in\R$, and this $T$ is called antiperiod.
The antiperiod $T$ of a well-periodic curve in Definition \ref{def:well-periodic} is uniquely determined as it is also related with the zero set of the curvature.
We stress that we have imposed continuity on the curvature (and $\Phi$) of a well-periodic curve; this is natural in view of our regularity hypothesis.
This continuity implies that $\Phi$ does not change the sign on each connected component of $\R\setminus T\Z$.
In addition, the combination of odd symmetry and periodicity implies the symmetry $\Phi(\frac{T}{2}+s)=\Phi(\frac{T}{2}-s)$, and also the sign-changing property $\Phi(s)\Phi(s+T)<0$ for $s\in\R\setminus T\Z$.

The following lemma exhibits some properties of well-periodic curves which we will use later.
The proof is elementary and safely omitted.

\begin{lemma}[Symmetry of well-periodic curves] \label{lem:0517}
Let $\gamma:[0,L]\to\R^2$ be a well-periodic planar curve and $k$ be the signed curvature of $\gamma$.
If $k$ has three points $0\leq s_1<s_2 < s_3 \leq L$ such that
\begin{align} \label{eq:lem0517} 
 k(s_i)=0 \ \ (i=1, 2, 3), \quad \text{and}\ \ k(s)\ne0  \ \  \text{for}  \ \ s \in(s_1, s_2) \cup (s_2, s_3),
\end{align}
then we have
\begin{align} 
&\gamma'(s_1+\sigma) = \gamma'(s_3+\sigma) \ \ \text{holds whenever} \ \ s_1+\sigma, \ s_3+\sigma\in[0,L], \label{eq:1022-1}\\
&k(\sigma+s_1) = -k(s_3-\sigma) \ \ \text{holds whenever} \ \ \sigma+s_1 , \ s_3-\sigma \in [0,L].  \notag
\end{align}
\end{lemma}

In addition, if a well-periodic curve $\gamma:[0,L]\to\R^2$ is $\frac{m}{2}$-fold in the sense of Definition \ref{def:m-fold}, its signed curvature $k$ satisfies that
\begin{align}\label{eq:0430-0}
k(s) = 0 \ \iff \ s=0, \ \frac{1}{m}L, \ \frac{2}{m}L, \ \ldots, \ L.
\end{align}
In particular, we may assume that $k=\Phi$ without loss of generality.

\begin{remark}
The definition of `$\frac{m}{2}$-fold' here is in line with the terminology of `$\frac{m}{2}$-fold figure-eight $p$-elastica' used in \cite[Definition 5.3]{MYarXiv2203}.
Note that $(p,m,r)$-arcs and $(p,m,r)$-loops defined in \cite[Definition 3.2]{MYarXiv2209} are also $\frac{m}{2}$-fold well-periodic curves (see Figure \ref{fig:Maple2} for $p=2$ and $m=1,2,3$).
\end{remark}

\section{Rigidity under the clamped boundary condition} \label{sect:clamped}

In this section we prove Theorems \ref{thm:planar_global} and \ref{thm:planar_local}.

\subsection{Rigidity for minimality}\label{subsect:minimal_clamped}

We first prove Theorem~\ref{thm:planar_global}.
The key idea is already mentioned in Section~\ref{sect:trick} and Figure \ref{fig_3}.

\begin{proof}[Proof of Theorem~\ref{thm:planar_global}]
We argue by contradiction.
Suppose that $\gamma$ is a global minimizer of $\F$ in $\Ac$, and has two points $s_1, s_2 \in [0,L]$ satisfying \eqref{eq:0831-2}.
Up to reversing the parametrization, we may assume that $s_1$ may be equal to $0$ but $s_2<L$.
Let us define a continuous function $\bar{\gamma}:[0,L]\to\R^2$ by 
\begin{align*}
\bar{\gamma}:=
\gamma|_{[0,s_1]} \oplus R \Big(\gamma|_{[s_1, s_2]}\Big) \oplus \gamma|_{[s_2,L]}, 
\end{align*}
where $R$ denotes the affine transformation describing the rotation through 180 degrees  around the point
\[
c:=\frac{1}{2} \Big(\gamma(s_1) +\gamma(s_2) \Big).
\]
In particular, we see that
\[
\bar{\gamma}(s) = -\gamma(-s+s_1+s_2)+2c, \quad  s\in [s_1, s_2].
\]
By \eqref{eq:0831-2} and by definition of $\bar{\gamma}$ we have $\bar{\gamma} \in C^1([0,L];\R^2)$, particularly noting that
\[
\lim_{s\downarrow s_1}\bar{\gamma}'(s) = \gamma'(s_2)= \gamma'(s_1), \quad 
\lim_{s\uparrow s_2} \bar{\gamma}'(s) = \gamma'(s_1)= \gamma'(s_2). 
\] 
We also notice that the signed curvature $\bar{k}$ of $\bar{\gamma}$ is 
\begin{align*} 
\bar{k}(s)= 
\begin{cases}
-k(-s+s_1+s_2), \quad & s\in(s_1, s_2), \\
k(s), &s\in [0,L] \setminus [s_1, s_2],
\end{cases}
\end{align*}
which is bounded and hence $\bar{\gamma} \in W^{2, \infty}(0,L;\R^2)\subset W^{2,p}(0,L; \R^2)$.
Since $\bar{\gamma}$ satisfies the same clamped boundary condition as $\gamma$ (even if $s_1=0$), we have $\bar{\gamma} \in \Ac$.
Moreover, the above formula for $\bar{k}$ implies $\F(\bar{\gamma})=\F(\gamma)$.
Hence $\bar{\gamma} $ is also a global minimizer of $\F$ in $\Ac$.
Then it follows from \eqref{H1*} that $\bar{\gamma}\in C^2([0,L];\R^2)$.
On the other hand,  we see that by definition of $\bar{k}$,
\[ 
\lim_{s\uparrow s_2} \bar{k}(s) = -k(s_1), \quad 
\lim_{s\downarrow s_2} \bar{k}(s) = k(s_2).
\]
This together with \eqref{eq:0831-2} implies that 
$\bar{\gamma} \not\in C^2(0,L;\R^2)$.
This is a contradiction.
\end{proof}

\begin{remark} \label{rem:optimal-global}
Theorem \ref{thm:planar_global} is optimal in the sense that if $s_2-s_1=L$ then the assertion may fail. 
In fact, in the case of $P_0=P_1$ and $V_0=V_1$ (i.e., closed curves), a one-fold circle is a global minimizer of $\mathcal{B}_2$ in $\Ac$ but its endpoints satisfy \eqref{eq:0831-2}.
\end{remark}

\subsection{Rigidity for stability}\label{subsect:stable_clamped}

Now we turn to the proof of Theorem \ref{thm:planar_local}.
We first give an abstract criterion based on a contradiction argument, which clarifies how to deduce instability.
As in the proof of Theorem \ref{thm:planar_global}, a key point is that we allow equality in condition (ii) below.

\begin{lemma} \label{thm:Magic}
Suppose \eqref{H1**} holds.
Let $\gamma\in\Ac$ 
and suppose that there exists a sequence $\{\gamma_j\}_{j\in \N} \subset \Ac$ satisfying the following three conditions:
\begin{align} \tag{C} \label{seq}
\begin{cases}
\mathrm{(i)} \ \ \gamma_j \to \gamma  \quad \text{in}\quad W^{2,p}(0,L;\R^2) \ \  \text{as} \ \  j\to \infty, \\
\mathrm{(ii)} \ \  \F(\gamma_j) \leq \F(\gamma) \quad \text{for all (large)}\quad j\in \N, \\
\mathrm{(iii)} \ \ \gamma_j \notin C^2(0,L;\R^2) \quad \text{for all (large)} \quad j\in \N.
\end{cases}
\end{align}
Then $\gamma$ is {\em not} a local minimizer of $\F$ in $\Ac$.
\end{lemma}
\begin{proof}
Let $\gamma\in\Ac$ satisfy the three conditions in \eqref{seq}.
We argue by contradiction, so suppose that $\gamma$ is a local minimizer of $\F$ in $\Ac$.
Then there is $\vd>0$ such that
\begin{align}\label{Yeq:0417.1}
\F(\gamma) \leq \F(\xi) \quad \text{for} \quad \xi\in\Ac \ \text{ with } \  \| \gamma - \xi \|_{W^{2,p}} < \vd.
\end{align} 
By \eqref{seq}-(i), there exists $j_0\in\N$ such that $\| \gamma- \gamma_j \|_{W^{2,p}} < \vd/2$ for $j\geq j_0$.
Henceforth we consider only $j\geq j_0$.
It follows from \eqref{seq}-(ii) and \eqref{Yeq:0417.1} that 
\begin{align*}
\F(\gamma_j) \leq \F(\gamma) \leq \F(\xi) \quad \text{for} \quad \xi\in\Ac \ \text{ with } \  \| \gamma_j - \xi \|_{W^{2,p}} < \frac{\vd}{2}.
\end{align*} 
This implies that $\gamma_j$ are also local minimizers of $\F$ in $\Ac$ for all $j\geq j_0$.
Therefore, $\gamma_j \in C^2([0,L]; \R^2)$ follows from \eqref{H1**}, 
which contradicts \eqref{seq}-(iii).
This completes the proof.
\end{proof}

In order to show Theorem~\ref{thm:planar_local}, 
we now apply the above criterion to prove the instability of any well-periodic curve $\gamma$ that `strictly' contains three inflection points $s_1<s_2<s_3$.
Our perturbation $\gamma_j$ here is constructed as follows:
Noting that $\gamma$ has symmetry centered at $s_2$ as in Figure \ref{fig_2} (1), we first cut the curve $\gamma$ out at the perturbed points $s_1+\frac1j$ and $s_3+\frac1j$ (with large $j$) and then rotate it through $180$ degrees around the middle of the cut points as in Figure \ref{fig_2} (2).
By symmetry the perturbed curve $\gamma_j$ is close to the original one $\gamma$.
It turns out that by periodicity and our rotation procedure, the perturbed curve $\gamma_j$ has discontinuous curvature at the cut points.

\begin{figure}[htbp]
\centering
\includegraphics[width=12cm]{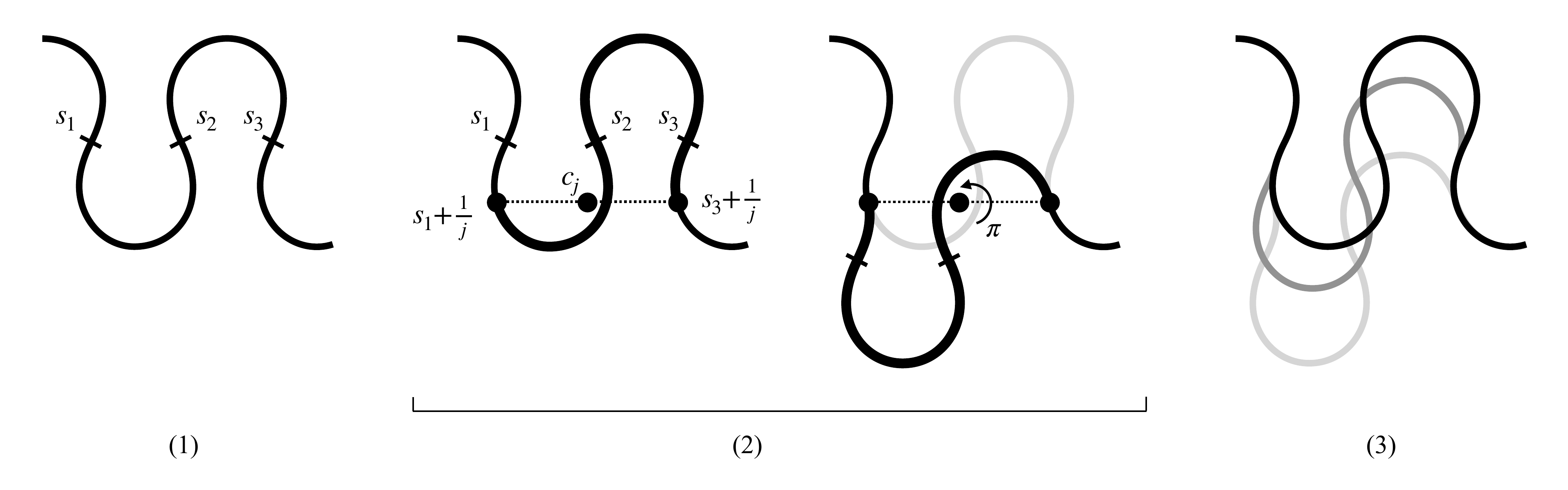} 
\caption{(1) An original well-periodic curve $\gamma$.
(2) Construction of a perturbation $\gamma_j$.
(3) Convergence as $j\to\infty$.
}
\label{fig_2}
\end{figure}

\begin{proof}[Proof of Theorem~\ref{thm:planar_local}]
Let $\gamma\in \Ac$ be a well-periodic curve and have three points satisfying \eqref{Yeq:0417.2}.
Up to reversing the parametrization and taking the first three inflection points, we may assume that $s_1,s_2,s_3$ satisfy \eqref{eq:lem0517}, in particular $s_3<L$.
In addition, up to reflection, we may assume that $k>0$ in $(s_1,s_2)$.
Then since $\gamma$ is well-periodic, $k<0$ in $(s_2, s_3)$ and $k>0$ in $(s_3,s_3+\delta)$ with some small $\delta\in(0,L-s_3)$.
For each integer $j\geq j_0$ with some $j_0$ such that $1/j_0<\delta$, we define a continuous function ${\gamma}_j : [0,L] \to \R^2$ by
\begin{align} \label{eq:0419.1}
\gamma_j:=\gamma|_{[0,s_1+\frac{1}{j}]}\oplus R_j\Big(\gamma|_{[s_1+\frac{1}{j}, s_3+\frac{1}{j}]}\Big) \oplus \gamma|_{[s_3+\frac{1}{j},L]}, 
\end{align}
where $R_j$ denotes the rotation through 180 degrees around 
\[ c_j := \frac{1}{2}\Big( \gamma(s_1+\tfrac{1}{j}) + \gamma(s_3+\tfrac{1}{j} ) \Big). \]
In particular, we see that 
\[
\gamma_j(s) = -\gamma(-s+s_1+s_3+\tfrac{2}{j}) + 2c_j, \quad s\in I_j:=(s_1 +\tfrac{1}{j}, s_3 +\tfrac{1}{j}). 
\]
In fact $\gamma_j$ is still a unit-speed $C^1$-curve by construction and by \eqref{eq:1022-1} in Lemma~\ref{lem:0517}; for example, around $s_1+\tfrac{1}{j}$, by computing the one-sided limits independently and using \eqref{eq:1022-1}, we deduce that $\gamma_j'(s_1+\tfrac{1}{j})=\gamma'(s_1+\tfrac{1}{j})$, and also the derivative $\gamma_j'$ is continuous by \eqref{eq:1022-1}, e.g.\ in view of
\begin{align*}
\lim_{s \downarrow s_1+\frac{1}{j} } \gamma_j' (s)
=\lim_{s \downarrow s_1+\frac{1}{j} } \gamma' (-s+s_1+s_3+\tfrac{2}{j}) 
= \gamma'(s_3+\tfrac{1}{j}) = \gamma'(s_1+\tfrac{1}{j});
\end{align*}
the same argument works for $s_3+\tfrac{1}{j}$.
Hence $\gamma_j \in C^1([0,L];\R^2)$ for any $j\geq j_0$.
In addition, by \eqref{eq:0419.1} the signed curvature $k_j$ of $\gamma_j$ is given by
\begin{align} \label{eq:0428-1}
k_j(s) =
\begin{cases}
-k(-s+s_1+s_3+\frac{2}{j}), \quad &s\in I_j \\
k(s), &s\in [0,L] \setminus \bar{I}_j, 
\end{cases}
\end{align}
where $\bar{I}_j:=[s_1 +\frac{1}{j}, s_3 +\frac{1}{j}]$.
Since $k\in C([0,L])$ by Definition~\ref{def:well-periodic}, 
we have $k_j\in L^{\infty}(0,L)$ and hence $\gamma_j \in W^{2, \infty}(0,L; \R^2)\subset W^{2,p}(0,L; \R^2)$ for all $j\geq j_0$.
Moreover, by \eqref{eq:0419.1} and the fact that $s_1+\frac{1}{j}>0$ and $s_3+\frac{1}{j}<L$, it is clear that the curve $\gamma_j$ satisfies the same boundary condition as $\gamma$.
Thus we have $\{\gamma_j\}_{j\geq j_0} \subset \Ac$.

Now we prove that the sequence $\{\gamma_j\}_{j\geq j_0}$ satisfies all the conditions in \eqref{seq}.

We first check \eqref{seq}-(i).
Recall that by Lemma~\ref{lem:0517},
\[
k(s) = -k(-s+s_1+s_3), \quad s\in I_j.
\]
This together with \eqref{eq:0428-1} implies that $k_j \to k$ a.e.\ in $(0,L)$.
In addition, noting that $\|k_j\|_{L^\infty} \leq \|k\|_{L^\infty}$, we obtain $k_j \to k$ in $L^p(0,L)$.
From this convergence and the fact that $\gamma(0)=\gamma_j(0)$ and $\gamma'(0)=\gamma_j'(0)$, we deduce that $\gamma_j \to \gamma$ in $W^{2,p}(0,L;\R^2)$.

Next we check \eqref{seq}-(ii).
Since
\begin{align*}
\int_{s_1+\frac{1}{j}}^{s_3+\frac{1}{j}} f\big(| k_j(s) | \big) \,ds
&=\int_{s_1+\frac{1}{j}}^{s_3+\frac{1}{j}} f\big(| k(-s +s_1 +s_3 +\tfrac{2}{j}) | \big) \,ds \\
&=\int_{s_3+\frac{1}{j}}^{s_1+\frac{1}{j}} - f\big(| k(s) | \big) \,  ds
= \int_{s_1+\frac{1}{j}}^{s_3+\frac{1}{j}} f\big(| k(s) | \big)   \,ds
\end{align*}
and since $\gamma$ and $\gamma_j$ agree elsewhere, we have $\F(\gamma_j) =\F(\gamma) $.

Finally, the discontinuity of curvature in \eqref{seq}-(iii) follows since for any large $j$, thanks to $k|_{(s_1,s_2)}>0$, we have
\[
\lim_{s \uparrow s_1+\frac{1}{j}} k_j(s)= k(s_1+\tfrac{1}{j})>0, \quad
\lim_{s \downarrow s_1+\frac{1}{j}} k_j(s)= -k(s_3+\tfrac{1}{j})=-k(s_1+\tfrac{1}{j})<0,
\]
and hence $k_j$ is not continuous at $s=s_1+\frac{1}{j}$.
The proof is complete.
\end{proof}

\begin{remark} \label{rem:M0503}
Since $ \Ac \subset \Ap$, Theorem~\ref{thm:planar_global} (resp.\ Theorem \ref{thm:planar_local}) holds true with $\Ac$ replaced by $\Ap$ whenever hypothesis \eqref{H1*} (resp.\ \eqref{H1**}) holds true for all unit vectors $V_0, V_1$.
\end{remark}

\section{Rigidity under the pinned boundary condition}\label{sect:pinned}

This section is devoted to the proofs of Theorems \ref{thm:pinned2} and \ref{thm:pinned1*}.

\subsection{Instability of one-fold waves}

We first prove Theorem \ref{thm:pinned2}.
In the same spirit as Lemma~\ref{thm:Magic}, it is sufficient to construct $\{\gamma_j\}_{j\in\N} \subset \Ap$ satisfying 
\begin{align} \tag{C$_{\rm p}$} \label{seq:pin}
\begin{cases}
\mathrm{(i)} \ \ \gamma_j \to \gamma  \quad \text{in}\quad W^{2,p}(0,L;\R^2) \ \  \text{as} \ \  j\to \infty, \\
\mathrm{(ii)} \ \  \F(\gamma_j) \leq \F(\gamma) \quad \text{for all (large)}\quad j\in \N, \\
\mathrm{(iii)} \ \ \text{the curvature}\  k_j \ \text{of} \ \gamma_j \  \text{satisfies} \ k_j(0)\ne0  \quad \text{for all (large)} \quad j\in \N.
\end{cases}
\end{align}
Here we construct such a perturbation $\gamma_j$ by extending the original curve by using its intrinsic periodicity and shift the domain as in Figure \ref{fig_4} so that the curvature does not vanish at the endpoints:

\begin{figure}[htbp]
\centering
\includegraphics[width=12cm]{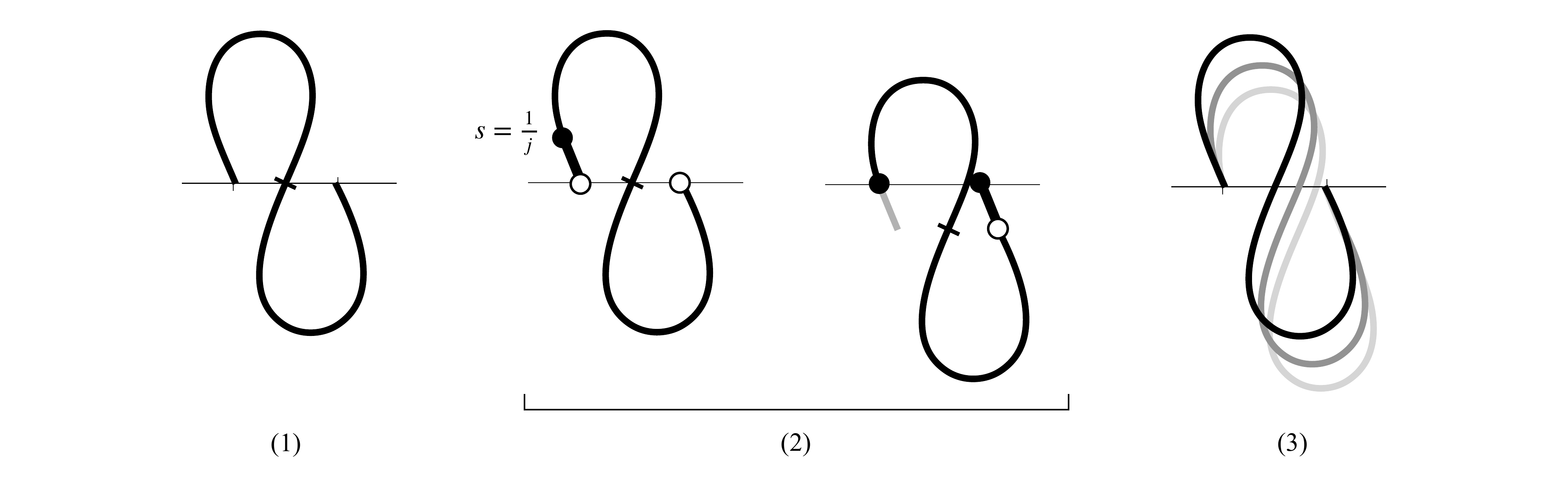} 
\caption{(1) A $1$-fold well-periodic curve $\gamma$.
(2) Construction of a perturbation $\gamma_j$.
(3) Convergence as $j\to\infty$.}
\label{fig_4}
\end{figure}

\begin{proof}[Proof of Theorem \ref{thm:pinned2}]
Recalling \eqref{eq:0430-0}, we infer that the curvature $k$ of $\gamma$ satisfies 
\begin{align} \label{eq:0430-1}
k(s) = 0 \ \iff \ s=0, \ \tfrac{L}{2}, \ L.
\end{align}
For each integer $j\geq j_0$ with some $j_0$ such that $1/j_0<L$, we define a continuous map $\gamma_j : [0,L] \to \R^2$ by
\begin{align}\label{eq:0430-2}
\gamma_j := P_0\oplus\gamma|_{[\frac{1}{j},L]}\oplus\gamma|_{[0,\frac{1}{j}]}. 
\end{align}
Then we infer from Lemma~\ref{lem:0517} that $\gamma'(0)=\gamma'(L)$ and hence
\[ 
\lim_{s \downarrow L-\frac{1}{j} } \gamma_j'(s)= \gamma' (0)=  \gamma' (L)= \lim_{s \uparrow L-\frac{1}{j} }  \gamma_j'(s),
\]
which implies that $\gamma_j \in C^1([0,L];\R^2)$ for each $j\geq j_0$.
The signed curvature $k_j$ of $\gamma_j$ is 
\begin{align} \label{eq:1fold-kj}
k_j (s)
:=
\begin{cases}
k(s+\tfrac{1}{j}) \quad &s\in[0,L-\tfrac{1}{j}), \\
k(s-L+\tfrac{1}{j})  \quad &s\in(L-\tfrac{1}{j}, L].
\end{cases}
\end{align}
Since $\lim_{s \uparrow L-\frac{1}{j}}k_j(s)=\lim_{s \downarrow L-\frac{1}{j}}k_j(s)=0$ holds by \eqref{eq:0430-1}, 
$k_j$ is continuous in $[0,L]$, and hence $\gamma_j \in C^2([0,L];\R^2)$.
By definition we have $\gamma_j (0)=P_0$, $\gamma_j(L)=P_1$, and $\mathcal{L}[\gamma_j]=L$.
Thus we have $\{\gamma_j\}_{j\geq j_0} \subset \Ap$.

Now we prove that the sequence $\{\gamma_j\}_{j\geq j_0}$ satisfies all the conditions in \eqref{seq:pin}. 

First, we check \eqref{seq:pin}-(i).
It follows from \eqref{eq:1fold-kj} that $k_j \to k$ a.e.\ in $(0,L)$.
Combining this with the fact that $\|k_j\|_{L^{\infty}}=\|k\|_{L^{\infty}}$, we see that $k_j\to k$ in $L^p(0,L)$.
Since $\gamma_j(0)=\gamma(0)$ and $\gamma_j'(0)=\gamma'(\frac{1}{j})\to\gamma'(0)$, we have $ \gamma_j \to \gamma$ in $W^{2,p}(0,L; \R^2)$.

Next, we check \eqref{seq:pin}-(ii).
It follows that
\begin{align}\notag
\begin{split}
\F(\gamma_j) &= \int_0^{L-\frac{1}{j}} f\big( |k_j(s) | \big) \,ds 
+ \int_{L-\frac{1}{j}}^L f\big( |k_j(s) | \big) \,ds \\
&= \int_0^{L-\frac{1}{j}} f\big( |k(s+\tfrac{1}{j}) | \big) \,ds + \int_{L-\frac{1}{j}}^L f\big( |k(s-L+\tfrac{1}{j}) | \big) \,ds 
 = \F(\gamma),
 \end{split}
\end{align}
and hence $\{ \gamma_j \}_{j\geq j_0} $ satisfies $\F(\gamma_j)= \F(\gamma)$.

Finally, we show \eqref{seq:pin}-(iii), i.e., the curvature at an endpoint does not vanish.
In fact, it follows from \eqref{eq:0430-1} and \eqref{eq:0430-2} that
$k_j(0) = k(\tfrac{1}{j}) \ne0$ for any $j\geq j_0$.

Thus $\{\gamma_j\}_{j\geq j_0}$ satisfies all the conditions in \eqref{seq:pin}.
Therefore, if $\gamma$ were a local minimizer, then so were all $\gamma_j$ with large $j$, but this would contradict \eqref{H2}.
The proof is complete.
\end{proof}

\subsection{Instability of loops}

Here we prove Theorem \ref{thm:pinned1*}.
Throughout this subsection, for notational simplicity, without loss of generality we assume that
\begin{equation}\label{eq:endpoints_xaxis}
    \text{$P_0=(0,0)$ and $P_1=(l,0)$, where
$l:=|P_0-P_1| \in [0,L).$}
\end{equation}
To begin with, we note here that the antiperiodicity of the curvature of well-periodic curves yields the following symmetry of curves; this is closely related with the symmetry $\Phi(\frac{T}{2}+s)=\Phi(\frac{T}{2}-s)$ which we have already observed in Section \ref{sect:pre}.

\begin{lemma} \label{lem:curve-sym}
Let $\gamma=(x,y):[0,L]\to\R^2$ be a $\frac{1}{2}$-fold well-periodic curve.
If $\gamma(0)=(0,0)$ and $\gamma(L)=(l,0)$ with some $l\neq 0$, then 
\[
x(L-s) + x(s) = l, \quad y(L-s)= y(s), \quad \text{for} \ \ s \in[0,L].
\]
\end{lemma}

\begin{proof}
    This easily follows from elementary differential geometry with the fact that
    \begin{align} \label{eq:1024-1}
        k(s)  = \Phi(s) =  - \Phi(s-L) = \Phi(L-s) = k(L-s)\quad \text{for any} \ \ s\in[0, L]. 
    \end{align}
    The assumption $l\neq0$ is used for forcing the reflection axis to be vertical.
\end{proof}

In addition, again for notational simplicity, we also assume that
$$f(0)=0$$
by replacing $f(t)$ with $f(t)-f(0)$ if necessary; this does not lose generality since the functional $\int_\gamma f(|k|)\,ds$ is (locally) minimized if and only if so does $\int_\gamma (f(|k|)-f(0) )\,ds$ under the fixed-length constraint.
Note that if $f$ satisfies \eqref{H3} and $f(0)=0$, then one easily verifies that
\begin{align}\label{eq:1007-1}
\text{$\lambda f(\lambda^{-1} t)  < f(t)$ for any $t\in (0,\infty)$ and $\lambda>1$,}
\end{align}
since $f((1-\lambda^{-1})0+\lambda^{-1}t)<(1-\lambda^{-1})f(0)+\lambda^{-1}f(t)$.

In what follows we will construct an energy-decreasing perturbation.
The key idea here is to perform an odd extension of the loop and shift the domain so that the symmetric loop is asymmetrically perturbed as in Figure \ref{fig_5} (3), and finally rescale the loop as in Figure \ref{fig_5} (4) in order to increase the shortened length and recover the admissibility.
All these procedures decrease the energy.
The convexity hypothesis will be used in the last rescaling step.

\begin{figure}[htbp]
\centering
\includegraphics[width=12cm]{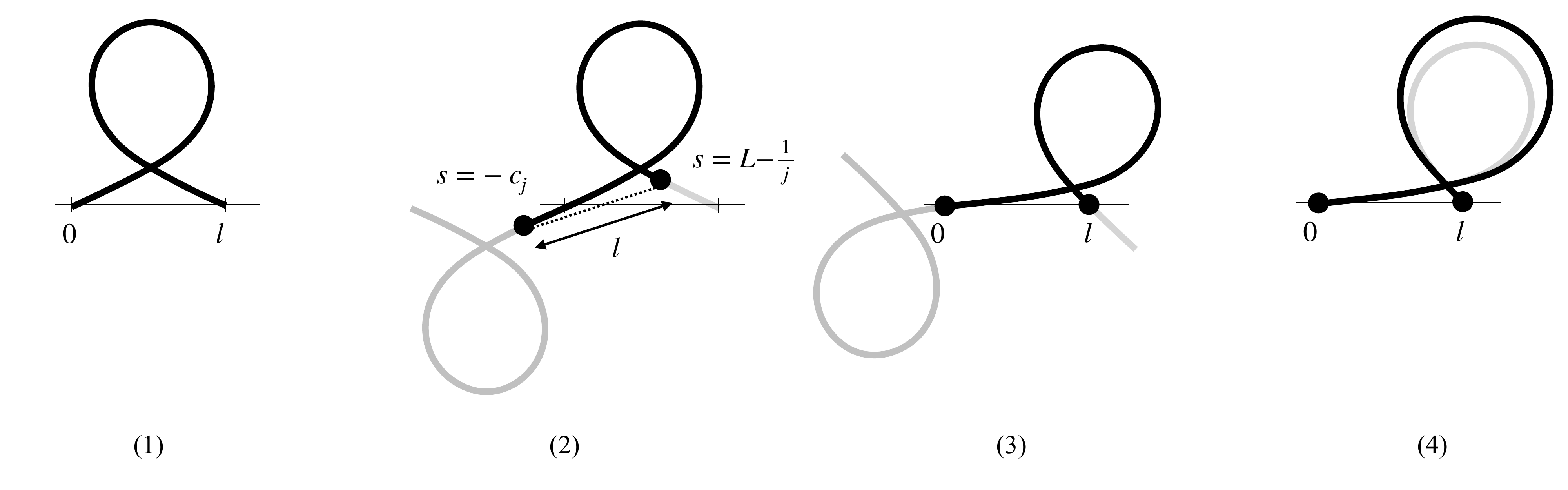} 
\caption{
(1) An original curve with a loop.
(2) Shifting the endpoints.
(3) Rigid motion for the boundary condition.
(4) Rescaling the loop for the length constraint.
}
\label{fig_5}
\end{figure}

\begin{proof}[Proof of Theorem \ref{thm:pinned1*}]
We first note that by assumption \eqref{eq:endpoints_xaxis} with $l>0$ we have the reflection symmetry in Lemma \ref{lem:curve-sym}, and in addition by $\gamma'(0)\cdot(P_1-P_0)>0$ we have
\begin{equation}\label{eq:0901-4}
    x'(0)>0.
\end{equation}
We also note that the non-injectivity assumption precisely means that
\begin{align}
    \text{there are distinct $a,b \in (0,L)$ such that $\gamma(a)=\gamma(b)$.} \label{eq:0901-2} 
\end{align}
In the proof below, we shall construct a family of curves 
$\{\gamma_j\}_{j\geq j_0} \subset \Ap$ satisfying 
\begin{align} \label{eq:0902-4}
\F(\gamma_j) < \F(\gamma) \ \ \text{for any} \ \ j\geq j_0 \quad \text{and} \quad  \gamma_j \to \gamma \ \ \text{in}\ \ W^{2,p}(0,L;\R^2).
\end{align}

First, we define an odd extension of $\gamma$ to the domain $[-L, L]$ by
\begin{align} \label{eq:0902-1}
\begin{cases}
\gamma(s) \quad &s \in [0, L], \\
-\gamma(-s) \quad &s \in [-L, 0), 
\end{cases}
\end{align}
which is also denoted by $\gamma: [-L, L] \to \R^2$.
Then $\gamma \in C^1([- L, L];\R^2)$ holds since $\gamma$ is of class $C^1$ around $s=0$.
For $j\geq \frac{1}{L}$, we have $y(\tfrac{1}{j})=y(L-\tfrac{1}{j})$ by Lemma \ref{lem:curve-sym}.
In addition, since $x'(0)>0$ by \eqref{eq:0901-4} and also $x'(L)>0$ by Lemma~\ref{lem:curve-sym}, for all large $j$ we have $x(L-\tfrac{1}{j})-x(\tfrac{1}{j})<x(L)-x(0)=l$ so that
\begin{align} \label{eq:1007-2}
 \big| \gamma(L- \tfrac{1}{j}) - \gamma(0) \big| <l.
 \end{align}
On the other hand, we infer from Lemma~\ref{lem:curve-sym} that 
\[
x(L-\tfrac{1}{j}) - x(-\tfrac{1}{j}) = x(L-\tfrac{1}{j}) + x(\tfrac{1}{j}) = l,
\quad 
y(L-\tfrac{1}{j}) = y(\tfrac{1}{j}),
\]
and since $|k(s)|\neq0$ for any small $s>0$, we have $y(\tfrac{1}{j})\neq0$ for all large $j$.
Hence
\begin{align}\label{eq:1013-1}
 \big| \gamma(L- \tfrac{1}{j}) - \gamma(-\tfrac{1}{j}) \big|  = \sqrt{l^2+4y(\tfrac{1}{j}) ^2} >l
\end{align}
for all $j\geq j_0$ with sufficiently large $j_0$.
By \eqref{eq:1007-2} and \eqref{eq:1013-1}, for each $j\geq j_0$ there is 
\begin{align} \label{eq:0508-1}
c_j \in (-\tfrac{1}{j}, \tfrac{1}{j}) \quad \text{such that}\quad \left| \gamma(L-\tfrac{1}{j}) - \gamma(-c_j) \right| = l.
\end{align}
(In fact we can show $c_j>0$ but this is not used.)
Define $\bar{\gamma}_j : [-c_j, L-\frac{1}{j}] \to \R^2$ by 
\[
\bar{\gamma}_j (u):= Q_j \Big( \gamma(u) - \gamma(-c_j) \Big), \quad u\in [-c_j, L-\tfrac{1}{j}] ,
\]
where $ Q_j $ is 
a rotation matrix such that $Q_j (\gamma(L-\frac{1}{j})- \gamma(-c_j)) = (l,0)$
and $Q_j\to \mathrm{Id}$ (by \eqref{eq:0508-1}, such $ Q_j $ exists).
Then we see that
\begin{align}\label{eq:0901-5}
\bar{\gamma}_j(-c_j)=(0,0), \quad \bar{\gamma}_j(L-\tfrac{1}{j})=(l,0), \quad  \bar{\gamma}_j \in W^{2,p}(-c_j, L-\tfrac{1}{j};\R^2)
\end{align}
but we still have $\mathcal{L}[\bar{\gamma}_j]<L$.
Now we normalize the length.
Let $a, b \in (0,L)$ satisfy \eqref{eq:0901-2}.
Then $\bar{\gamma}_j(a)= \bar{\gamma}_j(b)$ also holds for all $j\geq j_0$ with $j_0$ so large that $\tfrac{1}{j_0}< a<b<L-\tfrac{1}{j_0}$.
For $j\geq j_0$ we define a continuous function $\bar{\Gamma}_j : [-c_j, L-\tfrac{1}{j}] \to \R^2$ by
\begin{align*}
\bar{\Gamma}_j:= 
\gamma_j |_{[-c_j,a]} \oplus \lambda_j \gamma_j |_{[a,b]} \oplus \gamma_j |_{[b, L-\frac{1}{j}]}, \quad \lambda_j := \frac{b-a+\frac{1}{j}-c_j}{b-a} >1.
\end{align*}
Note that $\mathcal{L} [\bar{\Gamma}_j] =L$ follows by the choice of the scaling factor $\lambda_j$.

Hereafter we let $\gamma_j$ denote the arclength parametrization of $\bar{\Gamma}_j$.
We have $\{\gamma_j\}_{j\geq j_0}  \subset C([0,L];\R^2)$ by definition.
Since $\bar{\Gamma}_j$ is defined only by rescaling at a self-intersection point, we also have $\{\gamma_j\} \subset C^1([0,L];\R^2)$.
We also notice that the signed curvature $k_j$ of $\gamma_j$ is given by
\begin{align} \label{eq:0502-1}
k_j(s)= 
\begin{cases}
k(s-c_j), & 0 \leq s <  a+c_j,  \\
\lambda_j^{-1}k(\lambda_j^{-1} (s-a-c_j) +a) , \ \ & a+c_j< s < b+\frac{1}{j}, \\ 
k(s-\frac{1}{j}), &b +\frac{1}{j} < u \leq L,
\end{cases}
\end{align}
where $k$ is the signed curvature of $\gamma$.
Now we check that $\{\gamma_j \}_{j\geq j_0} \subset \Ap$.
We have $k\in C([0,L])$ by Definition~\ref{def:well-periodic}, and hence by \eqref{eq:0502-1} we get $\{\gamma_j\}_{j\geq j_0} \subset W^{2,\infty}(0,L;\R^2) \subset W^{2,p}(0,L;\R^2)$.
Moreover, $ \gamma_j(0)=(0,0)$ and $ \gamma_j(L)=(l,0)$ by definition and \eqref{eq:0901-5}.
Recalling $\mathcal{L}[\gamma_j]=L$, we have $\{\gamma_j \}_{j\geq j_0} \subset \Ap$.

Henceforth we show that the family $\{\gamma_j\}_{j\geq j_0}$ satisfies \eqref{eq:0902-4}.
We first show the strict inequality in \eqref{eq:0902-4}.
It follows from \eqref{eq:0502-1} and the change of variables that 
\begin{align} \label{eq:0902-2}
\begin{split}
\F(\gamma_j) 
&= \int_0^{a+c_j} f\big( |k_j(s) | \big) \,ds + \int_{a+c_j}^{b+\frac{1}{j}} f\big( |k_j(s) | \big) \,ds 
+ \int_{b+\frac{1}{j}}^L f\big( |k_j(s) | \big) \,ds \\
&= \int_{-c_j}^{a} f\big( |k(s) | \big) \, ds
+ \int_{a}^{b} f\big( \lambda_j^{-1} |k(s) | \big) \lambda_j \, ds
+\int_{b}^{L-\frac{1}{j}} f\big( |k(s) | \big) \,ds.
\end{split}
\end{align}
By \eqref{eq:1007-1} and $\lambda_j>1$, and the fact that $k\neq0$ in $(0,L)$ since $\gamma$ is $\frac{1}{2}$-fold, we obtain
\begin{align*}
 \int_{a}^{b} f\big( \lambda_j^{-1} |k(s) | \big) \lambda_j \, ds
<  \int_{a}^{b} f\big( |k(s) | \big)\, ds.
\end{align*}
By \eqref{eq:0902-1} we have $|k(s)|=|k(-s)|$ for $s\in[-\tfrac{1}{j},\tfrac{1}{j}]$, and hence by $|c_j|<\tfrac{1}{j}$ and $f\geq0$,
\begin{align*}
\int_{-c_j}^{a} f\big( |k(s) | \big) \,ds
&\leq \int_{0}^{|c_j|} f\big( |k(s) | \big) \,ds + \int_{0}^{a} f\big( |k(s) | \big) \,ds \\
&\leq \int_{0}^{\frac{1}{j} } f\big( |k(s) | \big) \,ds + \int_{0}^{a} f\big( |k(s) | \big) \,ds \\
&= \int_{L-\frac{1}{j} }^{L} f\big( |k(s) | \big) \,ds + \int_{0}^{a} f\big( |k(s) | \big) \,ds,
\end{align*}
where in the last equality we used $k(L-s)=k(s)$, cf.\ \eqref{eq:1024-1}.
Therefore, by \eqref{eq:0902-2},
\begin{align*} 
\begin{split}
\F(\gamma_j) 
&< \int_{0}^{a} f\big( |k(s) | \big) \, ds
+ \int_{a}^{b} f\big( |k(s) | \big) \, ds
+\int_{b}^{L} f\big( |k(s) | \big) \,ds 
=\F(\gamma). 
\end{split}
\end{align*}
It remains to show that $\gamma_j \to \gamma$ in $W^{2,p}(0,L;\R^2)$.
By \eqref{eq:0502-1} and the fact that $\lambda_j \to 1$, $c_j \to 0$, we see that $k_j \to k$ a.e.\ in $(0,L)$.
In addition, since $\|k_j\|_{L^\infty} \leq \|k\|_{L^\infty}$ holds by $\lambda_j >1$, it follows that $k_j \to k$ in $L^p(0,L)$.
Noting also the fact that $\gamma_j(0)=\gamma(0)$ and $\gamma_j'(0)=Q_j\gamma'(\frac{1}{j})\to\gamma'(0)$, we obtain $\gamma_j \to \gamma$ in $W^{2,p}(0,L;\R^2)$.
The proof is complete.
\end{proof}

\begin{remark}\label{rem:minimality_local}
Theorem~\ref{thm:pinned1*} will be used for showing instability of a $(p,r,1)$-loop with $r:=\tfrac{|P_0-P_1|}{L}>0$, but does not cover the case of $r=0$.
In fact, if $r=0$, then we can regard a $(p,r,1)$-loop as a half-fold figure-eight $p$-elastica, which is a global minimizer of $\mathcal{B}_p$ in $\Ap$ \cite[Theorem 1.4]{MYarXiv2209} and hence obviously stable.
Therefore it is necessary to assume that $P_0\neq P_1$ at least.
\end{remark} 

\section{Planar $p$-elasticae}\label{sect:p-elastica}

In this section we discuss the stability of closed and pinned planar $p$-elasticae.

\subsection{Closed $p$-elastica}\label{closed-elastica}

We first apply Theorem \ref{thm:planar_local} to prove Theorem \ref{thm:classify-closed-pela}, thus classifying the stability of closed planar $p$-elasticae.

Let $\Aclosed$ denote the set $\Ac$ in the special case that $P_0=P_1$ and $V_0=V_1$.
Recall from \cite[Theorem 5.6]{MYarXiv2203} that any closed planar $p$-elastica is either a circle or a figure-eight $p$-elastica, possibly multiply covered.

To begin with, we observe that an $m$-fold circle and a $1$-fold figure-eight $p$-elastica (in the sense of \cite{MYarXiv2203}) are indeed stable: 

\begin{proposition} \label{prop:closed-stable} 
If $\gamma$ is an $m$-fold circle, where $m\in\mathbf{N}$, or a $1$-fold figure-eight $p$-elastica, then $\gamma$ is a local minimizer of $\mathcal{B}_p$ in $\Aclosed$.
\end{proposition}
\begin{proof}
For $m\in \mathbf{N}\cup\{0\}$, let $Z_m\subset \Aclosed$ denote the subset of fixed rotation number $m$:
\[
Z_m:= \Set{\gamma \in \Aclosed | \mathcal{N}[\gamma]:=\tfrac{1}{2\pi} \textstyle \int_0^{L}k\,ds =m  }. 
\]
Since the functional $\mathcal{N}$ is constant in a small $W^{2,p}$-neighborhood of any element of $\Aclosed$, if $\gamma$ is a minimizer of $\mathcal{B}_p$ in $Z_m$, then $\gamma$ is a local minimizer in $\Aclosed$.

It suffices to show that an $m$-fold circle (resp.\ a 1-fold figure-eight $p$-elastica) is a minimizer of $\mathcal{B}_p$ in $Z_m$ if $m\geq1$ (resp.\ $m=0$).
The existence of a minimizer $\bar{\gamma}_m$ in $Z_m$ follows from the standard direct method (cf.\ \cite[Proposition 4.1]{MYarXiv2209}).  
Then, $\bar{\gamma}_m$ is also a $p$-elastica by the Lagrange multiplier method \cite{MYarXiv2209}. 
On the other hand, by the classification for closed $p$-elasticae \cite[Theorem 5.6]{MYarXiv2203}, in the case of $m\geq1$, $\bar{\gamma}_m$ must be an $m$-fold circle.
In the remaining case of $m=0$, the same classification also implies that any $p$-elastica with rotation number $0$ is an $n$-fold figure-eight $p$-elastica for $n\in \N$, and comparing their $p$-bending energy, we find that $\bar{\gamma}_0$ must be a $1$-fold figure-eight $p$-elastica.
\end{proof}

\begin{remark} \label{rem:optimal-local}
Here is a good position to observe that Theorem \ref{thm:planar_local} is optimal in the sense that if $s_3-s_1=L$ then the assertion may fail. 
In fact, a $1$-fold figure-eight $p$-elastica is a local minimizer of $\mathcal{F}=\mathcal{B}_p$ in $\Aclosed$ by the above proposition, while it is well-periodic and satisfies \eqref{Yeq:0417.2} at $s=0,\frac{L}{2},L$ if the endpoints are arranged to be located at the crossing of the figure-eight (inflection points).
\end{remark}

We are now ready to classify stable $p$-elasticae among closed curves.

\begin{proof}[Proof of Theorem~\ref{thm:classify-closed-pela}]
In view of \cite[Theorem 5.6]{MYarXiv2203} and Proposition~\ref{prop:closed-stable} it suffices to prove  instability of $m$-fold figure-eight $p$-elasticae for $m\geq2$.
If $m\geq2$, then $m$-fold figure-eight $p$-elasticae are well-periodic and have at least three inflection points satisfying \eqref{Yeq:0417.2}, and hence they are unstable in $\Aclosed$ by Theorem~\ref{thm:planar_local}.
\end{proof}

\begin{proof}[Proof of Corollary~\ref{cor:LangerSinger-p}]
As discussed in the proof of Proposition~\ref{prop:closed-stable}, if $m\geq1$, then any global minimizer with rotation number $m$ is an $m$-fold circle, which is stable in $\Aclosed$. 
If $m=0$, then it follows from Theorem~\ref{thm:classify-closed-pela} that any stable zero-rotation-number planar closed curve is a $1$-fold figure-eight $p$-elastica.
\end{proof}

\subsection{Wavelike pinned $p$-elastica}

Next we prove Theorem~\ref{thm:wavelike_unique} and Corollary \ref{cor:pinned_unique}, which are now almost direct consequences of Theorems~\ref{thm:planar_local}, \ref{thm:pinned2}, and \ref{thm:pinned1*}, combined with our previous results.

\begin{proof}[Proof of Theorem~\ref{thm:wavelike_unique}]
Let $\gamma$ be a wavelike pinned $p$-elastica and let $r:=\frac{|P_0-P_1|}{L}$. 
Then it follows from \cite[Theorem 1.1]{MYarXiv2209} that there exists $n\in \mathbf{N}$ such that $\gamma$ is either a $(p,r,n)$-arc or a $(p,r,n)$-loop, for which we write $\gamma^n_{\rm arc}$ or $\gamma^n_{\rm loop}$, respectively.
By \cite[Theorem 1.7]{MYarXiv2203}, every pinned $p$-elastica is of class $C^2$, and by \cite[Theorems 1.2 and 1.3]{MYarXiv2203},
the signed curvature of a wavelike $p$-elastica can be expressed by the so-called $p$-elliptic function $\cn_p$, which is an odd, antiperiodic, and continuous function. 
Therefore, $\gamma^n_{\rm arc}$ and $\gamma^n_{\rm loop}$ are well-periodic curves. 
Furthermore, 
\begin{itemize}
\item $\gamma_{\rm arc}^n$ and $\gamma_{\rm loop}^n$ ($n\geq 3$) are $\frac{n}{2}$-fold well-periodic curves and contain three inflection points satisfying \eqref{eq:lem0517},
\item $\gamma_{\rm arc}^2$ and $\gamma_{\rm loop}^2$ are $1$-fold well-periodic curves,
\item $\gamma_{\rm loop}^1$ is a $\frac{1}{2}$-fold well-periodic curve satisfying  \eqref{eq:0901-4} and \eqref{eq:0901-2} if $|P_0-P_1|>0$ (cf.\ \cite[Lemma 3.8]{MYarXiv2209} for existence of a loop).
\end{itemize}
Hence, $\gamma_{\rm arc}^n$ and $\gamma_{\rm loop}^n$ with $n\geq3$ are unstable by Theorem~\ref{thm:planar_local} with Remark~\ref{rem:M0503}, while so are $\gamma_{\rm arc}^2$ and $\gamma_{\rm loop}^2$ by Theorem~\ref{thm:pinned2}.
Moreover, if $|P_0-P_1|>0$, then $\gamma_{\rm loop}^1$ is unstable by Theorem \ref{thm:pinned1*}.
Therefore, in any case, the only remaining candidate for stable wavelike pinned $p$-elasticae is the global minimizer $\gamma_{\rm arc}^1$.
The proof is now complete.
\end{proof}

\begin{proof}[Proof of Corollary~\ref{cor:pinned_unique}]
If $p\in(1,2]$ or $p\in(2,\infty)$ and $\frac{|P_0-P_1|}{L}<\frac{1}{p-1}$, then any pinned $p$-elastica is a wavelike $p$-elastica (cf.\ \cite[Theorem 1.1]{MYarXiv2209}).
This together with Theorem~\ref{thm:wavelike_unique} asserts that any stable wavelike pinned $p$-elastica is a $(p,r,1)$-arc.
\end{proof}

\subsection{Flat-core $p$-elastica}\label{sect:flat}

Now we discuss some instability criteria for flat-core $p$-elasticae, and prove Theorem \ref{thm:unstable-flatcore}.
Flat-core $p$-elasticae appear as special examples of planar $p$-elasticae, and they are obtained by concatenating certain `loops' and `segments' with some arbitrariness (see \cite{nabe14, MYarXiv2203}).
In view of the natural boundary condition induced by the pinned boundary condition, we focus on flat-core $p$-elasticae each of whose endpoints is an endpoint of a segment or of a loop.
More precisely, we consider a curve $\gamma:[0,L]\to\R^2$ represented by, up to similarity and reparametrization, 
\begin{align}\tag{F1}\label{eq:F1}
\gamma 
= \gamma^{L_0}_{\rm seg} \oplus \bigg(\bigoplus_{j=1}^N\gamma^{\sigma_j}_{\rm loop} \oplus \gamma^{L_j}_{\rm seg} \bigg)
\end{align}
for some $N\in \N$, $\{\sigma_j\}_{j=1}^N\subset \{+, -\}$, and $L_1, \ldots, L_N\geq0$.
Here $\gamma_{\rm seg}^{L_j}(s):=(s,0)$ for $s\in[0,L_j]$, $\gamma_{\rm loop}^+$ denotes a certain loop, and $\gamma_{\rm loop}^-$ denotes the reflection of $\gamma_{\rm loop}^+$ with respect to the $e_1$-axis.
Although the explicit parametrization of $\gamma_{\rm loop}^+$ is known, here we only use the following special properties: the curve $\gamma_{\rm loop}^+=(X,Y):[0,1]\to\R^2$ is an arclength parametrized curve  such that the tangential vectors at both the endpoints are rightward, i.e.,
\begin{align}
    (\gamma_{\rm loop}^+)'(0)= (\gamma_{\rm loop}^+)'(1) = (1,0),  \tag{F2}\label{eq:F2}
\end{align}
the signed curvature $k_{\rm loop}$ of $\gamma_{\rm loop}^+$ satisfies
\begin{align}
    k_{\rm loop}>0 \ \ \text{in} \ \ (0,1), \quad k_{\rm loop}(0)=k_{\rm loop}(1)=0,  \tag{F3}\label{eq:F3}
\end{align}
and $\gamma_{\rm loop}^+$ is reflectionally symmetric in the sense that
\begin{align}
    X(s)+X(1-s)=2X(\tfrac{1}{2}), \ \ Y(s)=Y(1-s) \ \ \text{for} \ \ s\in[0,1]. \tag{F4}\label{eq:F4}
\end{align}
If $\gamma$ is a pinned planar flat-core $p$-elastica, then these properties directly follow by our previous classification \cite[Theorem 1.1]{MYarXiv2209} and explicit formulae \cite[Theorem 1.3]{MYarXiv2203} (see also Figure~\ref{fig:flat-cores}).
In addition, combining \eqref{eq:F1} with \eqref{eq:F2}, we see that any curve $\gamma$ of the form \eqref{eq:F1} satisfies 
\begin{align}\label{eq:ConditionFromF12}
\gamma'(0)=\gamma'(L)=(1,0).
\end{align}

The first criterion ensures that all loops need to lie `strictly inside' for stability under the pinned boundary condition.

\begin{proposition}\label{prop:flat-unstable-pin}
Let $\gamma\in \Ap$ be a flat-core $p$-elastica. 
If a loop touches an endpoint, 
then $\gamma$ is not a local minimizer of $\mathcal{B}_p$ in $\Ap$.
\end{proposition}
\begin{proof}
Without loss of generality, we may assume that $\gamma(0)$ is an endpoint of a loop.
For each $j\geq j_0$ with some $1/j_0<L$, we define $\gamma_j:[0,L]\to \mathbf{R}^2$ by 
\begin{align*}
\gamma_j:=P_0\oplus \gamma|_{[\frac{1}{j},L]}\oplus \gamma|_{[0,\frac{1}{j}]},
\end{align*}
cf.\ Figure~\ref{fig:flatcore-unstable1}.
Then it is clear that $\gamma_j \in C([0,L];\R^2)$.
Property \eqref{eq:ConditionFromF12} ensures that $\gamma_j \in C^1([0,L];\R^2)$.
Since $\gamma \in C^2([0,L];\R^2)$, the signed curvature $k_j$ of $\gamma_j$ is bounded, in particular, $\gamma_j \in W^{2,\infty}(0,L; \R^2) \subset W^{2,p}(0,L; \R^2)$.
By definition we have $\gamma_j(0)=P_0$, $\gamma_j(L)=P_1$, and $\mathcal{L}[\gamma_j]=L$. 
Thus $\{\gamma_j\}_{j\geq j_0} \subset \Ap$ holds.

As in the proof of Theorem~\ref{thm:pinned2}, 
we see that $\{\gamma_j\}_{j\geq j_0}$ satisfies the conditions \eqref{seq:pin}-(i) and \eqref{seq:pin}-(ii).
By property \eqref{eq:F3}, the signed curvature $k_j$ of $\gamma_j$ satisfies $k_j(0)\neq0$, and hence the condition \eqref{seq:pin}-(iii) is satisfied. 
Thus $\{\gamma_j\}_{j\geq j_0}$ satisfies all the conditions in \eqref{seq:pin} and hence $\gamma$ is unstable. 
\end{proof}

\begin{center}
    \begin{figure}[htbp]
      \includegraphics[width=12cm]{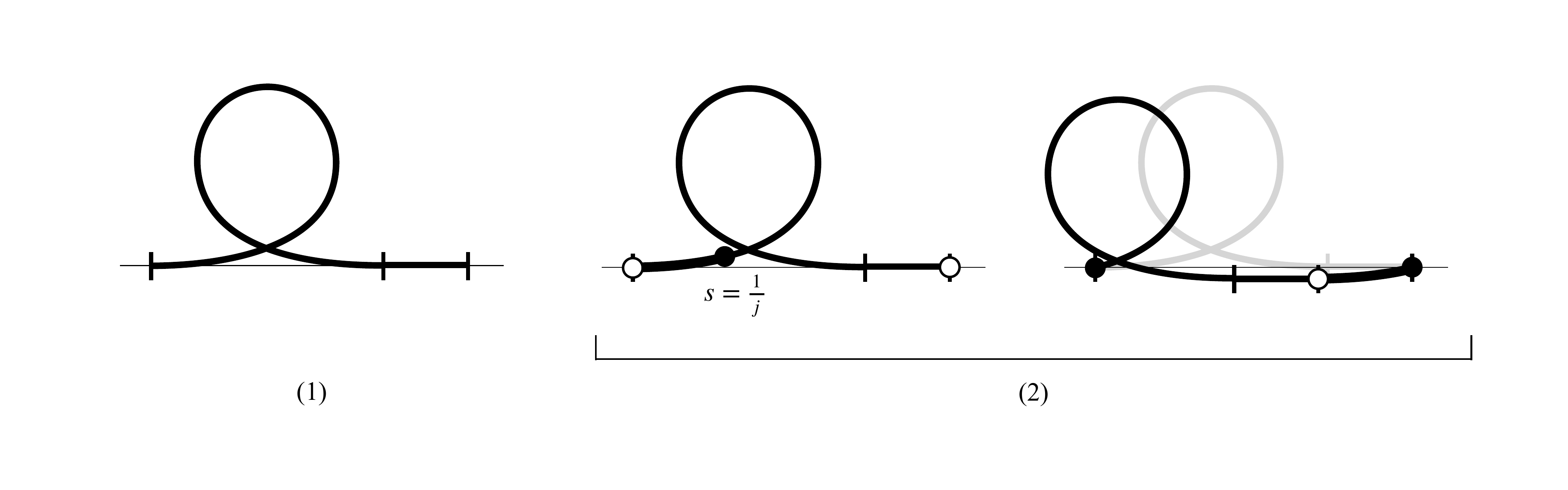}
  \caption{(1) A flat-core $p$-elastica with a loop touching an endpoint. (2) Construction of a perturbation $\gamma_j$.}
  \label{fig:flatcore-unstable1}
  \end{figure}
\end{center}

The next criterion ensures that a positive gap is necessary between two loops in opposite directions for stability even under the clamped boundary condition.

\begin{proposition}\label{prop:flat-unstable-clamp}
Let $\gamma\in \Ac$ be a flat-core $p$-elastica. 
If $\gamma$ contains one loop and a part of a loop in the opposite direction with no segment between the two loops, 
then $\gamma$ is not a local minimizer of $\mathcal{B}_p$ in $\Ac$.
\end{proposition}
\begin{proof}
Let $\gamma\in \Ac$ be a flat-core $p$-elastica satisfying the assumption.
This means that there is some $a\in(0,1]$ such that if we define
$$\Gamma:=\gamma^+_{\rm loop} \oplus \gamma^-_{\rm loop}|_{[0,a]},$$
then up to similarity and orientation of the parameter, the curve $\gamma$ can be decomposed as 
$$\gamma=\gamma|_{[0,s_1]}\oplus \Gamma \oplus \gamma|_{[s_2,L]}$$ 
for some $s_1, s_2\in[0,L]$ (where $L$ denotes the length after rescaling).
For each $j\geq j_0$ with some $j_0$ such that $1/j_0 < a$, we define 
\[
\Gamma_j:= \gamma^-_{\rm loop} |_{[0,\frac{1}{j}]} \oplus \gamma^+_{\rm loop} |_{[1-\frac{1}{j}, 1]} \oplus \gamma^+_{\rm loop} |_{[0,1-\frac{1}{j}]} \oplus \gamma^-_{\rm loop} |_{[\frac{1}{j}, a]}, 
\]
cf.\ Figure~\ref{fig:flatcore-unstable2}. 
Note that $\Gamma_j$ is a unit-speed $C^1$-curve by properties 
\eqref{eq:F2}, \eqref{eq:F4},
and \eqref{eq:ConditionFromF12}.
On the other hand, by property \eqref{eq:F3}, the curvature of $\Gamma_j$ is discontinuous at $s=1/j$ and hence $\Gamma_j$ is not of class $C^2$.

We define $\gamma_j:=\gamma|_{[0,s_1]}\oplus \Gamma_j \oplus \gamma|_{[s_2,L]}$.
Now we prove that the sequence $\{\gamma_j\}_{j\geq j_0}$ satisfies all the conditions in \eqref{seq}. 

We first check $\{\gamma_j\}_{j\geq j_0} \subset \Ac$. 
By definition, $\gamma_j$ is a $C^1$-curve and $\mathcal{L}[\gamma_j]=L$ holds.
Since $\gamma$ is a $C^2$-curve, the signed curvature of $\gamma_j$ is of bounded.
Therefore $\{\gamma_j\}_{j\geq j_0}$ is a $W^{2,p}$-curve. 
By construction, $\gamma_j$ satisfies the same boundary condition as $\gamma$.
Thus we have $\{\gamma_j\}_{j\geq j_0} \subset \Ac$.

The condition \eqref{seq}-(i) follows by straightforward calculations combined with the fact that the signed curvature of $\gamma^+_{\rm loop}$ is uniformly continuous.
The condition \eqref{seq}-(ii) follows since $\mathcal{B}_p[\Gamma]=\mathcal{B}_p[\Gamma_j]$.
The condition \eqref{seq}-(iii) follows from the fact that $\Gamma_j \notin C^2$. 
Consequently, by Lemma~\ref{thm:Magic} $\gamma$ is unstable in $\Ac$. 
\end{proof}

\begin{center}
    \begin{figure}[htbp]
      \includegraphics[width=12cm]{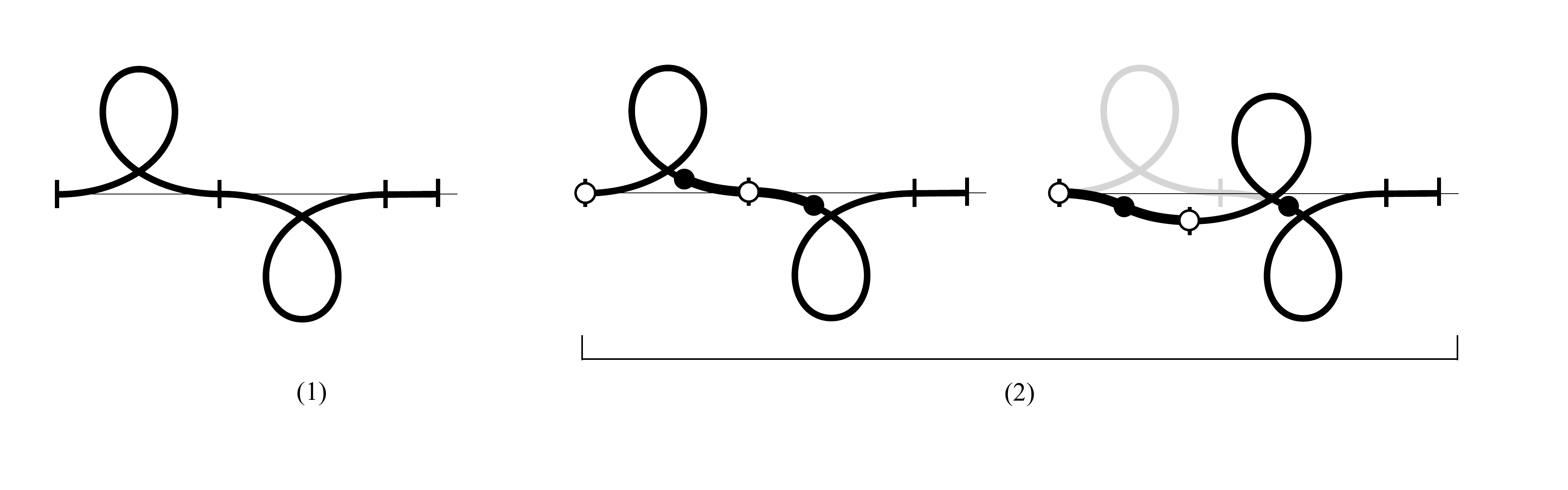}
  \caption{(1) A flat-core $p$-elastica with adjacent opposite loops. (2) Construction of a perturbation $\gamma_j$.}
  \label{fig:flatcore-unstable2}
  \end{figure}
\end{center}

The proof of Theorem~\ref{thm:unstable-flatcore} is now already complete.

\begin{proof}[Proof of Theorem~\ref{thm:unstable-flatcore}]
It directly follows by Propositions~\ref{prop:flat-unstable-pin} and \ref{prop:flat-unstable-clamp}.
\end{proof}

Now we introduce a new class of flat-core $p$-elasticae, which are not covered by the previous two instability results (and presumably stable). 
Let us first call a flat-core planar $p$-elastica $\gamma$ \emph{alternating} if up to similarity and reparametrization, $\gamma$ is given 
as in \eqref{eq:F1} with strictly positive $L_0, \ldots, L_N >0$.
The strict positivity particularly means that the segments and the loops appear alternately.
By relaxing the alternating property, we now define the following class:

\begin{definition}[Quasi-alternating]\label{def:alternating}
Let $\gamma$ be a flat-core planar $p$-elastica of the form \eqref{eq:F1} with some $N \in \N$, $\sigma_1, \ldots, \sigma_N\in \{+, -\}$, and $L_0, \ldots, L_N \geq0$. 
We call $\gamma$ \emph{quasi-alternating} if the following two conditions hold: 
\begin{itemize}
    \item[(i)] $L_0, L_N >0$.
    \item[(ii)] For $j\in\{1,\dots,N-1\}$, if $L_j=0$, then $\sigma_j=\sigma_{j+1}$.
\end{itemize}
\end{definition}

\begin{proof}[Proof of Corollary~\ref{cor:flatcore-dichotomy}]
Let $\gamma$ be a pinned planar $p$-elastica. 
Then, by \cite[Theorem 1.1]{MYarXiv2209}, $\gamma$ is a wavelike $p$-elastica or a flat-core $p$-elastica.
If $\gamma$ is wavelike and stable, then Theorem~\ref{thm:wavelike_unique} implies that $\gamma$ is a one-fold arc.
On the other hand, suppose that $\gamma$ is flat-core and stable.
By \cite{MYarXiv2209} any pinned flat-core $p$-elastica is of the form \eqref{eq:F1}.
Proposition~\ref{prop:flat-unstable-pin} and Proposition~\ref{prop:flat-unstable-clamp} imply conditions (i) and (ii) in Definition \ref{def:alternating}, respectively, and hence $\gamma$ must be quasi-alternating.
\end{proof}

\section{Spatial elasticae}\label{sect:spatial}

Finally we discuss applications to spatial elasticae, i.e., critical points of the bending energy
$\B[\gamma]:=\int_\gamma|\kappa|^2ds$, where $\kappa:=\gamma_{ss}$,
among fixed-length $W^{2,2}$-curves in $\R^3$.
Spatial elasticae are also smooth by a standard bootstrap argument so that our trick is applicable well.
A classification of spatial elasticae is obtained by Langer--Singer (cf.\ \cite{Sin}).

Here we obtain new rigidity principles due to the presence of spatial perturbations.
All the results are concerning the admissible space $\Ac$ of curves in $W^{2,2}_\mathrm{arc}(0,L;\R^3)$ subject to the standard clamped boundary condition.

\begin{theorem}[Rigidity for spatially minimal elasticae]\label{thm:spatial_minimal}
    Let $\gamma\in\Ac$.
    If there are $0\leq s_1<s_2\leq L$ with $s_2-s_1<L$ such that all the vectors $\gamma(s_2)-\gamma(s_1)$, $\gamma'(s_1)$, $\gamma'(s_2)$ are orthogonal to a unit vector $\omega\in\R^3$, and such that either $\gamma''(s_1)\cdot\omega\neq0$ or $\gamma''(s_2)\cdot\omega\neq0$, then $\gamma$ is not globally minimal.
\end{theorem}

\begin{proof}
    Suppose on the contrary that $\gamma$ is a global minimizer.
    Then $\gamma$ is of class $C^2$ in particular.
    Without loss of generality we may assume that $\gamma''(s_1)\cdot\omega\neq0$.
    We decompose $\gamma=\gamma|_{[0,s_1]}\oplus\gamma|_{[s_1,s_2]}\oplus\gamma|_{[s_2,L]}$.
    Let $P$ be a unique plane passing through $\gamma(s_1)$ and orthogonal to $\omega$.
    By the assumption the plane $P$ also passes through $\gamma(s_2)$ and is parallel to $\gamma'(s_1)$ and $\gamma'(s_2)$.
    Let $R$ denote the reflection about the plane $P$, and let $\tilde\gamma:=\gamma|_{[0,s_1]}\oplus R\gamma|_{[s_1,s_2]}\oplus\gamma|_{[s_2,L]}$.
    Then $\tilde\gamma\in\Ac$ by the property of $P$, and also $\B[\tilde\gamma]=\B[\gamma]$, so that $\tilde\gamma$ is also a global minimizer and hence of class $C^2$.
    However $\tilde\gamma''$ has discontinuity at $s_1$ since $\gamma''(s_1)\cdot\omega\neq0$ so that the vectors $\tilde\gamma''(s_1-0)=\gamma''(s_1)$ and $\tilde\gamma''(s_1+0)=\gamma''(s_1)-2(\gamma''(s_1)\cdot\omega)\omega$ do not coincide.
    This is a contradiction.
\end{proof}

\begin{theorem}[Rigidity for spatially stable elasticae]\label{thm:spatial_stable}
    Let $\gamma\in\Ac$.
    If there are $0\leq s_1<s_2\leq L$ with $s_2-s_1<L$ such that all the vectors $\gamma(s_2)-\gamma(s_1)$, $\gamma'(s_1)$, $\gamma'(s_2)$ are parallel, and such that either $|\gamma''(s_1)|\neq0$ or $|\gamma''(s_2)|\neq0$, then $\gamma$ is not locally minimal.
\end{theorem}

\begin{proof}
    Suppose on the contrary that $\gamma$ is a local minimizer.
    Then $\gamma$ is of class $C^2$ in particular.
    Without loss of generality we may assume that $|\gamma''(s_1)|\neq0$.
    We decompose $\gamma=\gamma|_{[0,s_1]}\oplus\gamma|_{[s_1,s_2]}\oplus\gamma|_{[s_2,L]}$.
    Let $L$ be a unique line passing through $\gamma(s_1)$ and parallel to $\gamma'(s_1)$.
    By the assumption the line $L$ also passes through $\gamma(s_2)$ and is parallel to $\gamma'(s_2)$.
    Let $R_{\theta}$ denote a rotation about the axis $L$ through angle $\theta$, and let $\gamma_\theta:=\gamma=\gamma|_{[0,s_1]}\oplus R_\theta\gamma|_{[s_1,s_2]}\oplus\gamma|_{[s_2,L]}$.
    Then $\gamma_\theta\in\Ac$ by the property of $L$, and also $\gamma_\theta\to\gamma$ in $W^{2,2}$ as $\theta\to0$ with $\B[\gamma_\theta]=\B[\gamma]$, so that for any small $\theta$ the curve $\gamma_\theta$ is also a local minimizer and hence of class $C^2$.
    However $\gamma_\theta''$ has discontinuity at $s_1$ (whenever $
    \theta\neq0$) since $\gamma''(s_1)$ is orthogonal to $\gamma'(s_1)$ and hence to $L$, so that the non-zero vectors $\gamma_\theta''(s_1-0)$ and $\gamma_\theta''(s_1+0)$ form the angle $\theta\neq0$ due to $R_\theta$.
    This is a contradiction.
\end{proof}

Below we provide some concrete examples of applications of the above results. 

The first example is about the known fact that all closed planar elasticae except for a one-fold circle are unstable in $\R^3$.
This fact is shown by Langer--Signer \cite{LS_85} and also follows by Maddocks' analysis \cite{Maddocks1981,Maddocks1984}, both through explicit computations of second variations.
Theorem \ref{thm:spatial_stable} provides a geometric proof of this fact.

\begin{corollary}\label{cor:spatial-circle}
A twice (or more times) covered circle is unstable in the set of closed curves of fixed length in $\R^3$ (cf.\ Figure~\ref{fig:spatial-circle}).
\end{corollary}

\begin{center}
    \begin{figure}[htbp]
      \includegraphics[scale=0.2]{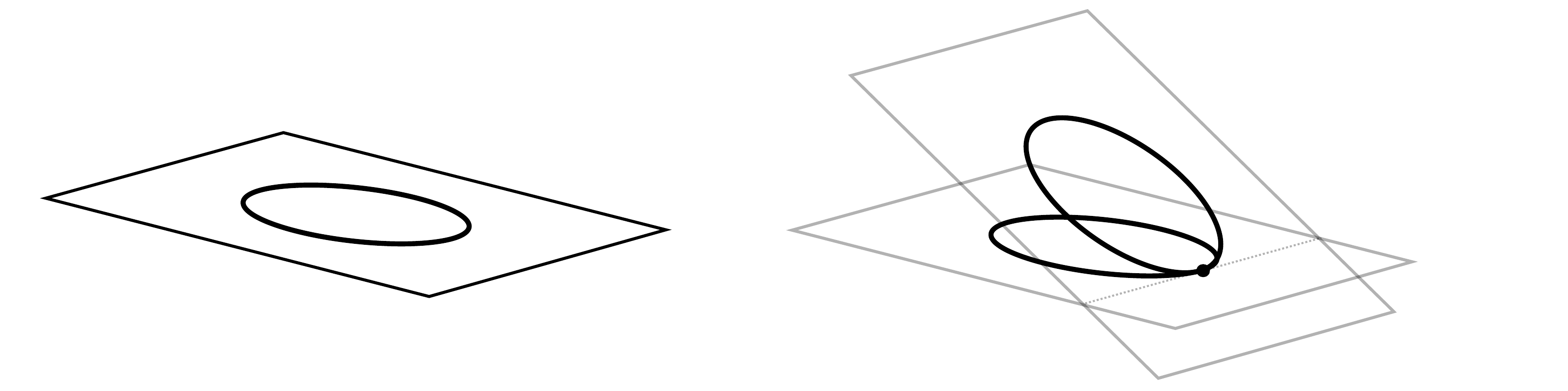}
  \caption{A more than two-fold circular elastica (left) and its perturbation for instability (right).}
  \label{fig:spatial-circle}
  \end{figure}
\end{center}

\begin{corollary}\label{cor:spatial-figure-8}
A once (or more times) covered figure-eight elastica is unstable in the set of closed curves of fixed length in $\R^3$ (cf.\ Figure \ref{fig:spatial-figure8}).
\end{corollary}

\begin{center}
    \begin{figure}[htbp]
      \includegraphics[scale=0.2]{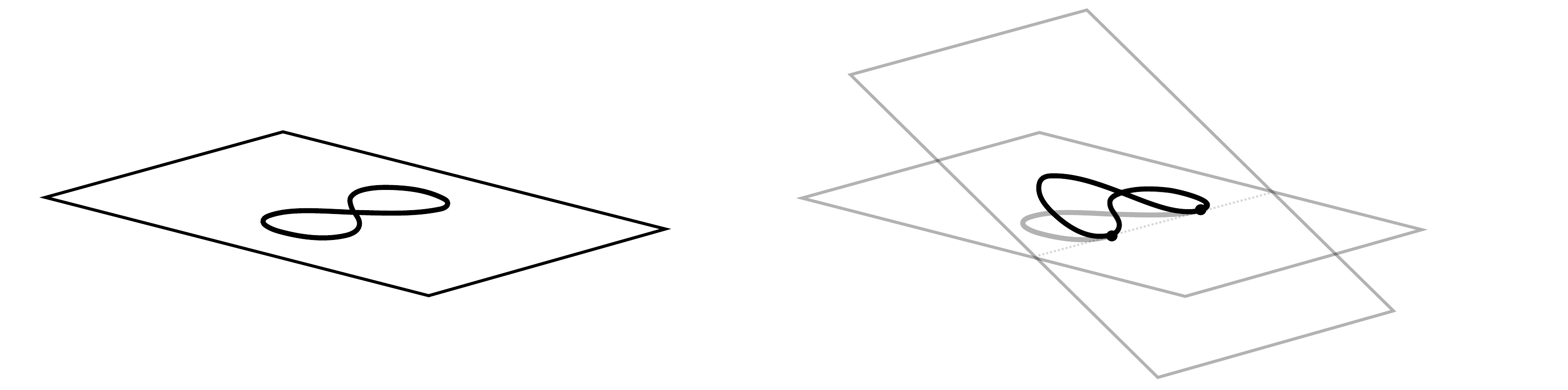}
  \caption{A one-fold figure-eight elastica (left) and its perturbation for instability (right).}
  \label{fig:spatial-figure8}
  \end{figure}
\end{center}

In the same way we can produce various examples of planar elasticae (both orbitlike and wavelike) which are stable in the plane but unstable in the space, cf.\ Figure~\ref{fig:spatial-wavelike}.
Some of such phenomena are also treated by Maddocks \cite{Maddocks1981,Maddocks1984}.

\begin{center}
    \begin{figure}[htbp]
      \includegraphics[scale=0.2]{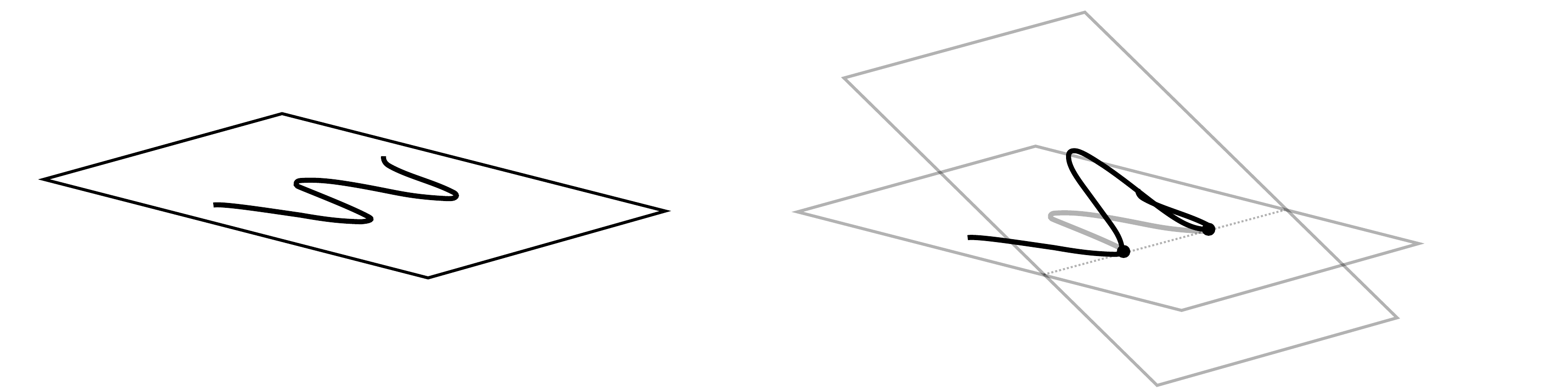}
  \caption{A wavelike elastica, to which Theorem~\ref{thm:spatial_stable} is applicable.}
  \label{fig:spatial-wavelike}
  \end{figure}
\end{center}

The final example is about a helix, an example of a purely spatial elastica.
By Langer--Singer \cite{LS_85} it is shown that if a helix has more than two turn, then it is unstable.
Theorem \ref{thm:spatial_minimal} gives a new rigidity for minimality.

\begin{corollary}\label{cor:spatial-helix}
A helical elastica with more than one turn is not minimal in the set of clamped curves of fixed length in $\R^3$ (cf.\ Figure~\ref{fig:spatial-helix}).
\end{corollary}

\begin{center}
    \begin{figure}[htbp]
      \includegraphics[scale=0.175]{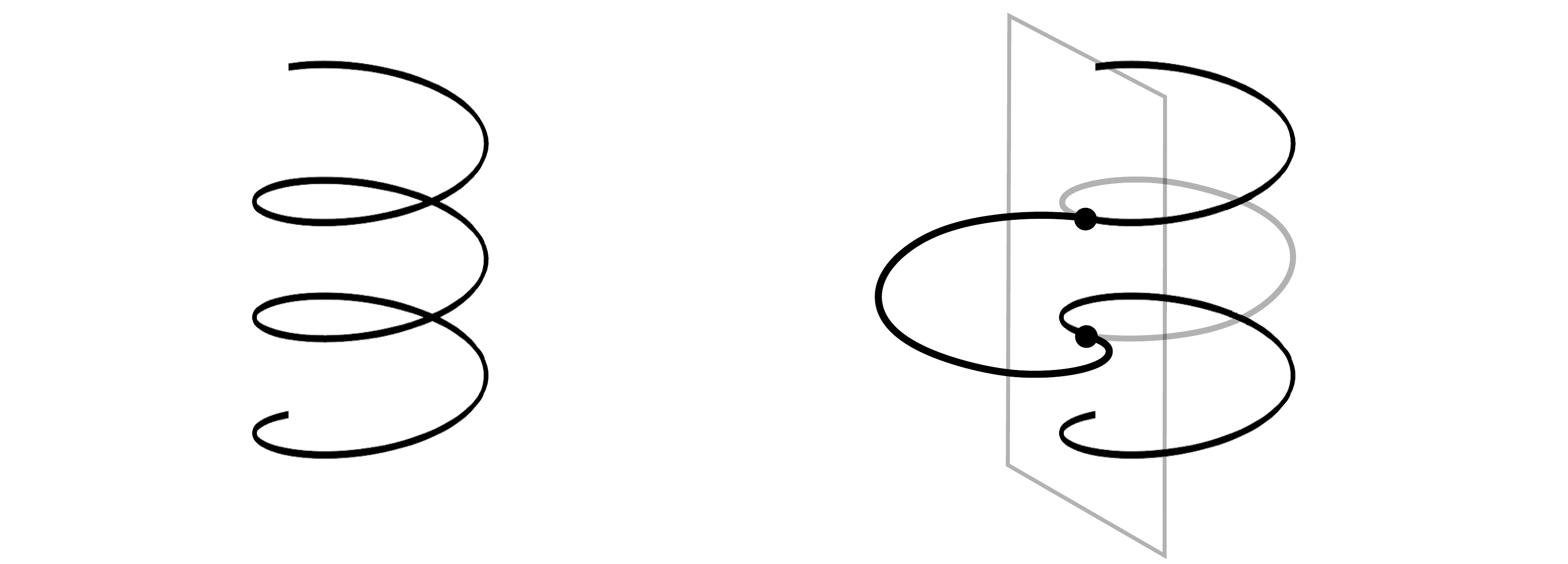}
  \caption{A helix (left) and its deformation for non-minimality (right).}
  \label{fig:spatial-helix}
  \end{figure}
\end{center}

We can also deduce by Theorem \ref{thm:spatial_minimal} that a spatial elastica is not minimal if it has a self-intersection (in its interior) at which one of the two curvature vectors are linearly independent from both the two tangent vectors.

We expect that our method also works for more general spatial ($p$-)elasticae or even other ambient spaces under some symmetry.

\appendix 

\section{Second variation for the $p$-bending energy}\label{app:second_variation}


In this section we formally compute the second variation of the $p$-bending energy $\B_p$ and observe that, in the singular regime $p\in(1,2)$, it is even possible that the second variation may not be well defined even at a critical point ($p$-elastica).
This fact suggests that the stability analysis of $p$-elasticae is substantially delicate compared to the classical case $p=2$.

First, we rigorously compute the second variation of the $p$-bending energy $\B_p$ under the absence of inflection points and the $C^2$-regularity of variations.
Recall from \cite[Theorem 1.7]{MYarXiv2203} that any planar $p$-elastica is of class $C^2$ and hence its signed curvature $k$ is continuous.
Let $T$ denote the unit tangent $T:=\gamma'$ and $N$ the unit normal such that $\gamma''=kN$.
Then we have the following

\begin{proposition}\label{prop:2nd-Gdiff}
Let $p\in(1,\infty)$, $L>0$, and $\gamma:[0,L]\to\R^2$ be an arclength parametrized $p$-elastica.
Suppose that the signed curvature $k\in C([0,L])$ of $\gamma$ has no zero.
Then the second variation $d^2\mathcal{B}_p$ of $\mathcal{B}_p$ at $\gamma$ in a direction $\eta\in C^{2}([0,L];\R^2)$ is given by
\begin{align}\label{eq:2nd-Gdiff}
\begin{split}
\langle d^2\B_p[\gamma] ,  \eta \rangle 
:&\!=\ \frac{d^2}{d\varepsilon^2}\B_p[\gamma+\varepsilon \eta] \Big|_{\varepsilon=0} \\
&=\ \mathbf{I}[\gamma](\eta)+\mathbf{J}[\gamma](\eta)+\mathbf{K}[\gamma](\eta),
\end{split}
\end{align}
where
$\mathbf{I}[\gamma](\eta)$, $\mathbf{J}[\gamma](\eta)$, and $\mathbf{K}[\gamma](\eta)$ are defined by
\begin{align}
\mathbf{I}[\gamma](\eta)&:=
\int_0^L |k|^p\Big[ (4p^2-1)(T, \eta')^2 +(1-2p)|\eta'|^2 \Big]\,ds, \notag
\\
\mathbf{J}[\gamma](\eta)&:=
\int_0^L  |k|^{p-1} \Big[ p(p-3)(N, \eta')^2 +2p(-2p+1)(N, \eta'')(T, \eta') \notag  \\
&\hspace{130pt}+2p(2p-3)(N, \eta')(T', \eta'') \Big]\,ds, \notag \\
\mathbf{K}[\gamma](\eta)&:= p(p-1)
\int_0^L |k|^{p-2} \Big[|\eta''|^2 +3(T, \eta'')^2 \Big]\,ds. \label{eq:2nd-Gdiff-K}
\end{align}

\end{proposition}
\begin{proof}
For $\eta \in W^{2,p}(0,L;\R^n)$ and small $\varepsilon\in\R$, define $\gamma_\varepsilon\in W^{2,p}(0,L;\R^n)$ by
\[
\gamma_{\varepsilon}(s) := \gamma(s) + \varepsilon \eta(s), \quad s\in[0,L].
\]
Then, the curvature vector $\kappa_{\varepsilon}$ of $\gamma_{\varepsilon}$ is given by 
\[
\kappa_{\varepsilon}= \frac{1}{| \gamma'_{\varepsilon}|^2 } \gamma''_{\varepsilon} - \frac{( \gamma'_{\varepsilon},  \gamma''_{\varepsilon})}{| \gamma'_{\varepsilon}|^4} \gamma'_{\varepsilon}.
\]
Hence, the $p$-bending energy of $\gamma_{\varepsilon}$ is represented by
\[ 
\B_p[\gamma_{\varepsilon}] = \int_{\gamma_{\varepsilon}} |\kappa_{\varepsilon} |^p   |\gamma'_{\varepsilon} | \,ds,
\]
and its derivative along $\{\gamma_\varepsilon\}_\varepsilon$ can be computed as
\begin{align*}
\frac{d}{d\varepsilon} \B_p[\gamma_{\varepsilon}] 
&= \int_0^L |\kappa_{\varepsilon} |^p\frac{ (\gamma'_{\varepsilon} , \eta' )}{| \gamma'_{\varepsilon} |} \,ds 
+ \int_0^L p | \kappa_{\varepsilon} |^{p-2} 
\big( \kappa_{\varepsilon} , \partial_{\varepsilon} \kappa_{\varepsilon}  \big) |\gamma'_{\varepsilon}|\, ds,
\\
\partial_{\varepsilon} \kappa_{\varepsilon}
&= \frac{1}{|\gamma'_{\varepsilon}|^2} \eta''  -2 \frac{(\gamma'_{\varepsilon} , \eta') }{| \gamma'_{\varepsilon} |^4} \gamma''_{\varepsilon}\\
&\quad- \bigg( \frac{(\gamma'_{\varepsilon}, \gamma''_{\varepsilon})}{| \gamma'_{\varepsilon} |^4} \eta' 
+\Big[ \frac{( \gamma'_{\varepsilon}, \eta'')+(\gamma''_{\varepsilon}, \eta') }{| \gamma'_{\varepsilon} |^4} -4 \frac{(\gamma'_{\varepsilon}, \gamma''_{\varepsilon}) (\gamma'_{\varepsilon}, \eta') }{| \gamma'_{\varepsilon} |^6}  \Big] \gamma'_{\varepsilon} \bigg).
\end{align*}
From this formula, the second derivative along $\{\gamma_\varepsilon\}_\varepsilon$ can also be computed as 
\begin{align*}
\frac{d^2}{d\varepsilon^2} \B_p[\gamma_\varepsilon] 
&= 2\int_0^L p|\kappa_\varepsilon|^{p-2}\big(\kappa_{\varepsilon}, \partial_\varepsilon \kappa_\varepsilon \big)\frac{ (\gamma'_{\varepsilon} , \eta' )}{| \gamma'_{\varepsilon} |} \,ds 
+ \int_0^L  | \kappa_{\varepsilon} |^{p} 
\bigg( \partial_\varepsilon \Big(\frac{ \gamma'_{\varepsilon}}{| \gamma'_{\varepsilon} |} \Big), \eta' \bigg)\, ds \\ 
&\quad + \int_0^L p(p-1) |\kappa_\varepsilon|^{p-2} \big|\partial_\varepsilon \kappa_{\varepsilon} \big|^2 |\gamma'_{\varepsilon}| \,ds 
+ \int_0^L  p| \kappa_{\varepsilon} |^{p-2} 
\big(\kappa_{\varepsilon}, \partial^2_\varepsilon \kappa_\varepsilon \big)|\gamma'_{\varepsilon}|\, ds.
\end{align*}
Noting that $|\gamma'|\equiv1$ and $(\gamma', \gamma'')=0$, we have
\begin{align*}
\partial_\varepsilon \kappa_\varepsilon \Big|_{\varepsilon=0}
&=\eta''-2(\gamma', \eta')\gamma''+\big( (\gamma', \eta'') + (\gamma'', \eta') \big) \gamma', \\
\big(\kappa_\varepsilon, \partial^2_\varepsilon \kappa_\varepsilon \big)\Big|_{\varepsilon=0}
&=-\,4{(\gamma',\eta')}(\gamma'',\eta'') + \big(-2 |\eta'|^2 +8(\gamma',\eta')^2 \big)|\gamma''|^2 \\
&\quad -2\big(-2(\gamma',\eta'') +8(\gamma'',\eta')\big)(\gamma'',\eta'). 
\end{align*}
Consequently, using $\gamma'=T$ and $\gamma''=\kappa_0 =kN$, we obtain the desired formula as in \eqref{eq:2nd-Gdiff}.
\end{proof}


\begin{remark}[Ill-definedness for $W^{2,p}$-variations: The singular case $p<2$]
We already notice that for $p\in(1,2)$ the second variation may diverge in the whole energy space $W^{2,p}$.
Indeed, $\mathbf{K}[\gamma](\eta)\geq p(p-1)k_*\int_0^L|\eta''|^2ds$ holds for $k_*:=\min|k|>0$, while $\mathbf{I}[\gamma](\eta)$ and $\mathbf{J}[\gamma](\eta)$ are always convergent by H\"older's inequality.
This means that the second variation formula in Proposition \ref{prop:2nd-Gdiff} makes sense for any $\eta\in W^{2,p}(0,L;\R^2)$ if and only if $p\geq 2$.
\end{remark}

The above fact suggests that in the singular regime $p\in(1,2)$ it is delicate to handle second variations.
However, if one focuses on instability, then it may be sufficient to look at smooth variations only (e.g.\ as in \cite{GPT23}).
For smooth variations $\eta$, Proposition \ref{prop:2nd-Gdiff} ensures that every term is well defined at least when the curvature $k$ has no zero.
Hence, the main issue is reduced to the case in the presence of zeroes of $k$, where the term $\mathbf{K}[\gamma](\eta)$ involving $|k|^{p-2}$ may diverge.
In \cite{MYarXiv2203} it is shown that the case of wavelike $p$-elasticae has the worst regularity.
For any arclength parametrized wavelike $p$-elastica, the signed curvature $k$ is given by $k(s)=a\cn_p(a(s-s_0),q)$ for some $a>0$, $q\in(0,1)$, and $s_0\in \R$ (see \cite[Definition 3.4]{MYarXiv2203} for $\cn_p$). 
Thus the integrability of $|k|^{p-2}$ is reduced to the following lemma.

\begin{lemma}\label{lem:cnp-L^(p-2)}
Let $p\in (1,\infty)$, $q\in (0,1)$, and $\cn_p(\cdot, q)$ be the $p$-elliptic cosine function.
Let $U$ be a neighborhood of zero of $\cn_p(\cdot, q)$. 
Then, 
\[
\cn_p(\cdot, q) \in L^{p-2}(U) \quad \text{if and only if} \quad p>\frac{3}{2}.
\]
\end{lemma}
\begin{proof}
First, recall that the $p$-elliptic cosine function $\cn_p(\cdot, q)$ is defined by 
\[
\cn_p(x,q):=|\cos \am_{1,p}(x,q)|^{\frac{2}{p}-1}\cos \am_{1,p}(x,q), 
\]
where $\am_{1,p}(\cdot,q)$ is the inverse function of 
\[
x\mapsto \mathrm{F}_{1,p}(x,q) := \int_0^x \frac{|\cos \theta|^{1-\frac{2}{p} } }{\sqrt{1-q^2 \sin^2\theta} }\,d\theta.
\]
Let $ \K_{1,p}(q):= \mathrm{F}_{1,p}({\pi}/{2},q)$.
Then, we see that the set of zero points of $\cn_p(\cdot, q)$ is given by
\[
Z_{p,q}:=\Set{ x\in \R | \cn_p(x,q)=0 } = \Set{ (2n+1) \K_{1,p}(q) }.
\]
By periodicity, it suffices to consider a neighborhood of $ \K_{1,p}(q)$. 
For a sufficiently small $\delta>0$, we have
\begin{align*}
\int_{ \K_{1,p}(q)-\delta}^{ \K_{1,p}(q)+\delta} |\cn_p(x,q)|^{p-2}\,dx
&=\int_{ \K_{1,p}(q)-\delta}^{ \K_{1,p}(q)+\delta} \Big( |\cos \am_{1,p}(x,q)|^{\frac{2}{p}}\Big)^{p-2}\,dx \\
&=\int_{ \frac{\pi}{2}-\delta'}^{ \frac{\pi}{2}+\delta'} \Big( |\cos y|^{\frac{2}{p}}\Big)^{p-2} \frac{|\cos y|^{1-\frac{2}{p}}}{\sqrt{1-q^2\sin^2 y}}\,dy,
\end{align*}
where in the last equality we used the change of variables $y=\am_{1,p}(x,q)$ and set $\delta'>0$ by $\frac{\pi}{2}+\delta'=\am_{1,p}(\K_{1,p}(q)+\delta,q)$.
Since $y\mapsto 1/\sqrt{1-q^2\sin^2 y}$ is continuous, the integrability of $|\cn_p(\cdot,q)|^{p-2}$ is reduced to that of $ |\cos y|^{3-\frac{6}{p}}$ around $y=\frac{\pi}{2}$.
This is equivalent to $3-\frac{6}{p}>-1$, and hence we obtain the desired conclusion. 
\end{proof}

\begin{remark}[Ill-definedness for smooth variations: The highly singular case $p\leq\frac{3}{2}$]
Now we find that the second variation for wavelike $p$-elasticae with $p\in (1,\frac{3}{2}]$ may not exist even if $\eta$ is smooth.
Since we still have $|\mathbf{I}[\gamma](\eta)|,|\mathbf{J}[\gamma](\eta)| < \infty$, again we only need to take care of the term $\mathbf{K}[\gamma](\eta)$.
Let $\gamma:[0,L]\to\R^2$ be an arclength parametrized wavelike $p$-elastica with some $z\in(0,L)$ such that $k(z)=0$.
Consider $\eta\in C^2(0,L;\R^2)$ such that $|\eta''(z)|\neq0$; then there is $c_*>0$ such that $|\eta''|\geq c_*$ holds in a neighborhood $U$ of $z$.
Then, 
by the change of variables we see that 
\begin{align} \nonumber
\mathbf{K}[\gamma](\eta) \geq p(p-1)c_*^2 \int_{U} |k(s)|^{p-2} \,ds =  p(p-1)c_*^2a^{p-3} \int_{\tilde{U}} |\cn_p(x,q)|^{p-2} \,dx,
\end{align}
where $\tilde{U}$ is a neighborhood of a zero of $\cn_p(\cdot, q)$.
Thus we deduce by Lemma~\ref{lem:cnp-L^(p-2)} that 
$\mathbf{K}[\gamma](\eta)$ diverges if (and in fact only if) $p\in(1,\frac{3}{2}]$.
\end{remark}

\subsection*{Acknowledgements}
The first author is supported by JSPS KAKENHI Grant Numbers 18H03670, 20K14341, and 21H00990, and by Grant for Basic Science Research Projects from The Sumitomo Foundation.
The second author is supported by JSPS KAKENHI Grant Number 22K20339.


\bibliographystyle{abbrv}
\bibliography{ref_Miura-Yoshizawa-ver5}

\end{document}